\documentclass{amsart}

\usepackage[colorlinks, linkcolor=blue,  citecolor=blue, urlcolor=blue]{hyperref}%
\usepackage{url}
\usepackage{verbatim}
\usepackage{alltt}
\usepackage{graphicx}
\graphicspath{{figures/}}
\usepackage{siunitx}
\usepackage{amsthm}

\newtheorem{theorem}{Theorem}[section]

\newtheorem{lemma}[theorem]{Lemma}
\newtheorem{proposition}[theorem]{Proposition}
\newtheorem{definition}[theorem]{Definition}
\newtheorem{remark}{Remark}[section]
\theoremstyle{definition}
\newtheorem{example}{Example}[section]

\newcommand{\norm}[1]{\Vert#1\Vert}
\newcommand{\Norm}[1]{\left\Vert#1\right\Vert}
\newcommand{\abs}[1]{\vert#1\vert}

\DeclareMathOperator{\trace}{trace}

\newcommand{\bq}{\begin{equation}}
\newcommand{\eq}{\end{equation}}
\newcommand{\R}{\mathbb{R}}
\newcommand{\Z}{\mathbb{Z}}

\newcommand{\Rn}{\R^n}

\newcommand{\bO}{\mathcal{O}}

\newcommand{\lam}{\lambda}
\newcommand{\sig}{\sigma_2}
\newcommand{\sigex}{\bar \sigma}
\DeclareMathOperator{\sort}{sort}
\DeclareMathOperator{\diag}{diag}
\newcommand{\x}{\textbf{x}}

\newcommand{\Dxx}{\mathcal{D}_{xx}}
\newcommand{\Dyy}{\mathcal{D}_{yy}}
\newcommand{\Dzz}{\mathcal{D}_{zz}}
\newcommand{\Dxy}{\mathcal{D}_{xy}}
\newcommand{\Dxz}{\mathcal{D}_{xz}}
\newcommand{\Dyz}{\mathcal{D}_{yz}}
\newcommand{\Dt}{\mathcal{D}}

\begin{document}

\title[Two-Hessian PDE]
{Numerical Methods for the 2-Hessian Elliptic Partial Differential Equation}

\author{Brittany D. Froese}
\thanks{Department of Mathematical Sciences, New Jersey Institute of Technology, University Heights, Newark, NJ 07102 USA ({\tt bdfroese@njit.edu})}

\author{Adam M. Oberman}
\thanks{Department of Mathematics and Statistics, McGill University, 805 Sherbrooke Street West, Montreal, Quebec, H3A 0G4, Canada ({\tt adam.oberman@mcgill.ca})}

\author{Tiago Salvador}
\thanks{Department of Mathematics and Statistics, McGill University, 805 Sherbrooke Street West, Montreal, Quebec, H3A 0G4, Canada ({\tt tiago.saldanhasalvador@mail.mcgill.ca}). This author partially supported by FCT doctoral grant SFRH / BD / 84041 /2012.}

\date{\today}



\begin{abstract}
The elliptic 2-Hessian equation is a fully nonlinear partial differential equation (PDE) that is related to intrinsic curvature for three dimensional manifolds.
We introduce two numerical methods for this PDE: the first is provably convergent to the viscosity solution, and the second is more accurate, and convergent in practice but lacks a proof.
The PDE is elliptic on a restricted set of functions:  a convexity type constraint is needed for the ellipticity of the PDE operator. Solutions with both discretizations are obtained using Newton's method. Computational results are presented on a number of exact solutions which range in regularity from smooth to nondifferentiable and in shape from convex to non convex.
\end{abstract}

\maketitle


\begin{figure}
\includegraphics[width=0.65\textwidth]{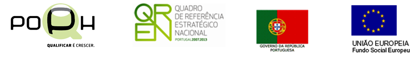}
\end{figure}

\tableofcontents

\section{Introduction}

In this article we study numerical approximations of a fully nonlinear elliptic partial differential equation (PDE), the $k$-Hessian equation.  The $k$-Hessian equations are a family of PDEs in $n$-dimensional space, which include the Laplace equation, when $k = 1$, and the Monge-Amp\`{e}re equation, when $k=n$.   
We have already studied the Dirichlet problem for the Monge-Amp\`{e}re equation \cite{ObermanFroeseMATheory,  ObermanFroeseFast,FroeseObermanFiltered}.   Here we study the first instance of this equation which is neither the Laplacian, or the  Monge-Amp\`{e}re equation, which is the 2-Hessian equation in three dimensions,
\bq\label{2Hessiandef}
S_2[u] =  u_{xx}u_{yy}+u_{xx}u_{zz}+u_{yy}u_{zz}-u_{xy}^2-u_{xz}^2-u_{yz}^2.
\eq

While the 2-Hessian equation is unfamiliar outside of Riemannian geometry and elliptic regularity theory, it is closely related to the  scalar curvature operator, which provides an intrinsic curvature for a three dimensional manifold.  
Geometric PDEs have been used widely in image analysis~\cite{sapiro2006geometric}. In particular, the Monge-Amp\`{e}re equation in the context of Optimal Transportation has been used in three dimensional volume based image registration~\cite{haker2004optimal}.  Scalar curvature equations have not yet been used in these contexts, perhaps because no effective solvers for PDEs involving this operator have yet been developed. The 2-Hessian operator also appears in conformal mapping problems. 
 Conformal surface mapping has been used for two dimensional image registration~\cite{angenent1999conformal,gu2004genus}, but does not generalize directly to three dimensions.  Quasi-conformal maps have been used in three dimensions~\cite{wang2007conformal,zeng2011registration}, however these methods are still being developed.   
 
In this article we introduce a monotone discretization of the $2$-Hessian equation in the three-dimensional case. A proof of convergence to the viscosity solution is provided. We also build a second order accurate finite difference solver which, while ustable if a simple iteration is used, can be modified to converge in practice.
 Numerical results are presented on solutions with varying regularity.
 
We focus on the Dirichlet problem
\[\begin{cases}
S_2[u] = f,& \quad \text{in } \Omega,\\
u = g, &\quad \text{on } \partial \Omega,
\end{cases}\]
where $\Omega$ is a rectangular (three dimensional box) domain, which is natural when treating computationally prescribed curvature problems. (For other topologies, different boundary conditions need to be used.  For the torus, periodic boundary conditions can be used. For the sphere, it is more complicated, but it is possible to patch together several cubic domains to obtain this topology.)

The operator is not elliptic, unless an additional constraint is imposed, which corresponds loosely
 to the requirement that the Laplacian restricted to every two-dimensional plane be positive.  This condition is explained in Proposition~\ref{prop2} and if we assume that $f > 0$, it reduces to
\[
\frac{d^2 u}{dv^2} +  \frac{d^2 u}{dw^2} \ge 0, \quad \text{ for every orthogonal triplet of vectors } (v,w,z).
\]
In other words, the two dimensional Laplacian restricted to every plane is positive for the function $u$. Hence the discretizations of the operator must also enforce the condition above. This means that either we are working with a family of inequality constraints, which makes the discretization very challenging, or that we need to find a way to encode the constraints in the PDE. We pursue the second option for the monotone discretization.

\subsection{Related work on curvature equations} 
The 2-Hessian equation is closely related to a  curvature PDE in three dimensions. 
In two dimensions there are several works on the evolution of curves using curvature, going back to the seminal paper of Osher and Sethian~\cite{Osher88frontspropagating}. In~\cite{ObermanMC}, a finite difference monotone scheme is given for the motion of level sets by mean curvature. The advantage of monotone discretizations is that they have a convergence proof, and convergent schemes are more stable and allow for faster solvers~\cite{Sethian95afast}. The surface evolver \cite{brakke1992surface} is a tool to evolve two dimensional surfaces by curvature based on the minimization of its energy. In \cite{sapiro2006geometric} one can find a relation between geometric PDEs and image analysis. For a review of the numerical methods for curvature flows see~\cite{ANU:298724}.

\subsection{Related work on the Monge-Amp\`{e}re equation}

In this paper we study a fully nonlinear elliptic PDE, while most of the curvature flows lead to quasilinear parabolic papers. Thus, we also review some of the related work on the Monge-Amp\`{e}re equation, a fully nonlinear elliptic PDE. For an extended review on numerical methods for fully nonlinear elliptic PDEs see \cite{FengNeilanReview}.

The Monge-Amp\`{e}re equation has been exhaustively studied. Consistent schemes using either finite elements \cite{NeilanQuadratic,NeilanFiniteElement} or finite differences \cite{MANewton} have been proposed. However, these schemes are not monotone and therefore do not fall within the convergence framework of Barles and Souganidis \cite{BSnum}. They require instead the PDE solution to be sufficiently smooth and the numerical solver to be well initialized. Using wide stencil discretizations, consistent monotone schemes were built \cite{ObermanFroeseMATheory,ObermanFroeseFast}, which are thus provably convergent but have limited accuracy due to their directional resolution. This limitation has been overcome recently. By introducing filtered schemes, which blend a monotone scheme with an accurate (but possibly unstable scheme), the authors in \cite{FroeseObermanFiltered} were able to obtain a provably convergent scheme with improved accuracy. Two other solutions, specific to particular dimensions, have been proposed as well: in the two dimensional setting using a mixture of finite differences and ideas from discrete geometry \cite{benamou2014monotone} and in the three dimensional setting using ideas from discretizations of optimal transport based on power diagrams \cite{Mirebeau3DMA}.



The Monge-Amp\`{e}re problem is related to the problem of prescribed Gauss curvature.  A numerical method for the problem of prescribed Gauss curvature can be found in \cite{merigot2014discrete}. The Gauss curvature flow is also used in image processing for surface fairing \cite{elsey2007analogue}.
There are very few publications devoted to solving it. In the early work of~\cite{sorensen2010quadratically} a quadratically constrained eigenvalue  minimization problem is solved to obtain the solution of the 2-Hessian equation. 
 
 \subsection{Scalar curvature and the 2-Hessian equation} 
The \emph{Gaussian curvature} of a two-dimensional surface is the product of the principal curvatures, $\kappa_1, \kappa_2$ of the surface.  It is an intrinsic quantity: it does not depend on the embedding of the surface in space.  Locally, the surface can be defined as the graph of a function $u(x)$, whose gradient of the function vanishes at $x$.  Then the Gaussian curvature at $x$ is given by the determinant of the Hessian of $u(x)$, 
\[
\det (D^2 u) = \kappa_1 \kappa_2,
\]
which is the two dimensional Monge-Amp\`{e}re operator applied to $u$ (if the gradient of $u$ does not vanish at $x$, additional first order terms appear).

The sign of the Gaussian curvature characterizes the surface,
and relates how the area of a geodesic ball in a curved Riemannian surface deviates from that of the standard ball in Euclidean space (larger or smaller depending on the sign).  The uniformization theorem of complex analysis establishes the fact that every surface has a conformal metric of constant Gaussian curvature: the sphere, the Euclidean plane, or hyperbolic space.   The uniformization theorem can be proved by several different methods. A natural method is one that solves a semi-linear Laplace equation  for the conformal map; see~\cite[Section 8]{mazzeo2002curvature}.

\emph{Curvature in three and higher dimensions} 
In general dimensions, curvature is a tensor rather than a scalar quantity.
The curvature tensor is defined by the sectional curvature, $K(p,x)$, which is given by the Gaussian curvature of the geodesic surface defined by the tangent plane, $p$, at x.
The scalar curvature (or the Ricci scalar),  which is the trace of the curvature tensor,  is the simplest curvature invariant of a Riemannian manifold.  It can be characterized as a multiple of the average of the sectional curvatures.  If we choose coordinates so that a three dimensional surface is given by the graph of a function $u(x)$ whose gradient vanishes at $x$, then the scalar curvature is given by a constant multiple of the 2-Hessian operator:
\[
\frac 1 2  \left ( \trace(D^2u )^2  - \trace\left((D^2u)^2\right) \right ) =  \kappa_1 \kappa_2  + \kappa_1 \kappa_3 + \kappa_2 \kappa_3
\]
where $\kappa_1, \kappa_2, \kappa_3$ are the three principal curvatures.  
Again, if the gradient of $u$ does not vanish at $x$, additional first order terms appear.  However  the equation above holds in general if we replace the principal curvatures with the eigenvalues of the Hessian.  This leads to the 2-Hessian equation; see \autoref{sec:2hessian} below.   

Since the second order terms pose the primary challenge in the solution of nonlinear elliptic equations, we focus on the 2-Hessian equation in this work. In a similar way, the Monge-Amp\`{e}re equation can be related to the equation for Gauss curvature through the inclusion of appropriate first order terms.  In~\cite{Benamou2014} we studied an extension of the Monge-Amp\`{e}re equation with first order nonlinear terms; in that case the primary challenge was the boundary conditions.

\subsection{{Differential geometry and $k$-Hessian equations}}
Conformal changes of metric (multiplication of the metric by a positive function) have played an important role in surface theory~\cite{lee1987yamabe}.  

One of the foundational problems of Riemannian differential geometry is to generalize the uniformization theorem for surfaces to higher dimensions. The generalization of the uniformization theorem for surfaces to higher dimensional manifolds involves replacing constant Gauss curvature (which is a scalar in two dimensions) with constant scalar curvature (rather than constant tensor curvature).  The resulting problem is called, 

\textbf{The Yamabe Problem} \textit{Given a compact Riemannian manifold $(M, g)$ of dimension $n \ge 3$, find a metric conformal to $g$ with constant scalar curvature.}

The solution of the Yamabe problem can be obtained by solving a nonlinear elliptic eigenvalue problem ~\cite{trudinger1968remarks}.  Generalizations of the Yamabe problem to other curvatures result in  $k$-Hessian type equations~\cite{viaclovsky2000some, viaclovsky1999conformal}.

Also of interest is

\textbf{The Calabi-Yau problem}~\cite{gross2003calabi} \textit{Find a conformal mapping, given by $U(x)$, which transforms a given metric $g_{ij}$ to a new one $\tilde g_{ij}$ given by }
\[
\tilde g_{ij} = \exp(U) g_{ij}.
\]

The conformal mapping function $U(x)$ satisfies a real Monge-Amp\`{e}re type PDE~\cite{MR480350}.
In certain settings (for example, the quaterionic setting), the Calabi-Yau problem for a manifold which is even ($d = 2n$) dimensional, results in a  $k$-Hessian type equation with $k = d/2$~\cite{alesker2010quaternionic}.

Another interesting problem where the $k$-Hessian equation appears is

\textbf{The Christoffel-Minkowski Problem} \textit{Find a convex hypersurface with the $k$-th symmetric function of the principal radii prescribed on its outer normals.}

It turns out that the solution of the Christoffel-Minkowski problem corresponds to finding convex solutions of a $k$-Hessian equation on the $n$-sphere \cite{Guan:2003aa}.

The 2-Hessian equation corresponds to scalar curvature, as we discuss above,  and solving the 2-Hessian PDE (or a related one) allows for the construction of hyper-surfaces of prescribed curvatures, for example scalar curvature \cite{guan2002convex}.

Also related are the problem of local isometric embedding of Riemannian surfaces in $\R^3$ and the related Weyl problem \cite{trudinger2008monge}.


\section{Background on the equation}

In this section,  we present the background analysis for the $k$-Hessian equation, with particular focus on the 2-Hessian equation in the three dimensional case. We follow the review by Wang \cite{Wang}.

The $k$-Hessian equation can be written as
\[S_k[u] = f\]
where $1\le k \le n$, $S_k[u] = \sigma_k(\lam(D^2u))$, $\lam(D^2u) = (\lam_1,\dots, \lam_n)$ are the eigenvalues of the Hessian matrix $D^2u$ and
\[
\sigma_k(\lam) = \sum_{i_1 < \dots < i_k} \lam_{i_1} \dots \lam_{i_k}
\]
is the $k$-th elementary symmetric polynomial. It includes the Poisson equation $(k=1)$
\[\Delta u = f,\]
and the Monge-Amp\`{e}re equation $(k=n)$
\[\det D^2u = f,\]
as particular cases.

The Dirichlet problem is given by
\bq\label{kHessian}\tag{kH}\begin{cases}
S_k[u] = f,& \quad \text{in } \Omega,\\
u = g, &\quad \text{on } \partial \Omega.
\end{cases}
\eq

\subsubsection*{Admissible functions and ellipticity}

When $k$ is even, the $k$-Hessian equation lacks uniqueness: if $u$ solves the $k-$Hessian equation, so does $-u$. Thus an additional condition is needed to ensure solution uniqueness. Moreover, when studying the Poisson equation it is customary to focus on the case $f \geq 0$, which is equivalent to look for solutions that are subharmonic since as a result a maximum principle holds. In the case of the Monge-Amp\`{e}re equation, we impose instead the additional constraint that $u$ is convex, which is required for the ellipticity of the equation. In either cases, it is thus necessary to restrict the solutions to an appropriate class of functions in order to ensure the equation has interesting properties.

Set
\[
\Gamma_k = \left\{\lambda \in \Rn \mid \sigma_j(\lambda) > 0, j=1,\ldots,k\right\}.
\]
$\Gamma_k$ is a symmetric cone, meaning that any permutation of $\lambda$ is in $\Gamma_k$. When $k=1$, $\Gamma_1$ is the half space $\left\{\lambda \in \Rn \mid \lambda_1+\ldots+\lambda_n > 0\right\}$. When $k=n$, $\Gamma_n$ is the positive cone $\Gamma_n = \left\{\lambda \in \Rn \mid \lambda_j > 0, j=1,\ldots,n\right\}$. The result is a restriction to subharmonic functions for $k=1$ and convex functions for $k=n$, as mentioned above.

\begin{definition}
A function $u \in C^2$ is $k-$admissible if $\lambda(D^2u) \in \overline{\Gamma_k}$.
\end{definition}

\begin{proposition}
If $u$ is $k-$admissible then the $k-$Hessian equation \eqref{kHessian} is (degenerate) elliptic.
\end{proposition}

\begin{remark}
We allow the eigenvalues of $u$ to lie in the boundary of $\Gamma_k$ and in such case the $k-$Hessian equation may become degenerate elliptic.
\end{remark}

\subsubsection*{Viscosity Solutions}\label{sec:viscosity}

Well-posedness and regularity for the equation is studied in \cite{caffarelli1985dirichlet}. Here we start by recalling a well posedness result.

\begin{definition}
We say that $\Omega \subseteq \R^n$ is $(k-1)$-convex if it satisfies
\[\sigma_{k-1}(\kappa) \geq c_0 > 0 \quad \text{ on } \partial \Omega\]
for some positive constant $c_0$ where $\kappa = (\kappa_1,\ldots,\kappa_{n-1})$ denote the principal curvatures of $\partial \Omega$ with respect to its inner normal.
\end{definition}

\begin{theorem}\label{wellposedness}
Assume that $\Omega$ is a bounded $(k-1)$-convex domain in $\R^n$ with $C^{3,1}$ boundary $\partial \Omega$, $g \in C^{3,1}\left(\partial \Omega\right)$ and $f \in C^{1,1}\left(\overline{\Omega}\right)$ with $f \geq f_0 > 0$. Then there is a unique $k$-admissible solution $u \in C^{3,\alpha}\left(\overline{\Omega}\right)$ to the Dirichlet problem \eqref{kHessian} for some $\alpha \in (0,1)$.
\end{theorem}

We now recall the definition of viscosity solutions.

\begin{definition}\label{def:visc}
A function $u \in USC\left(\overline{\Omega}\right)$ is a viscosity subsolution of \eqref{kHessian} if for every $\phi \in C^2\left(\overline{\Omega}\right) \cap \overline{\Gamma_k}$, whenever, $u-\phi$ has a local maximum at $x\in \overline{\Omega}$ then
\[\begin{cases}
\sigma_k(\lambda(D^2\phi(x))) \leq f	,				& \text{if } x \in \Omega,\\
\min\left(\sigma_k(\lambda(D^2\phi(x)))-f,u-g\right) \leq 0,	& \text{if } x \in \partial \Omega.
\end{cases}\]
Similarly, a function $u \in LSC\left(\overline{\Omega}\right)$ is a viscosity supersolution of \eqref{kHessian} if for every $\phi \in C^2\left(\overline{\Omega}\right) \cap \overline{\Gamma_k}$, whenever, $u-\phi$ has a local minimum at $x\in \overline{\Omega}$ then
\[\begin{cases}
\sigma_k(\lambda(D^2\phi(x))) \geq f	,					& \text{if } x \in \Omega,\\
\max\left(\sigma_k(\lambda(D^2\phi(x)))-f,u-g\right) \geq 0,	& \text{if } x \in \partial \Omega.
\end{cases}\]
Finally, we call $u$ a viscosity solution  of \eqref{kHessian} if $u^*$ is a viscosity subsolution and $u_*$ is a viscosity supersolution  of \eqref{kHessian}.
\end{definition}



The equations we consider satisfy a comparison principle.
\begin{align}\tag{CP}\label{CP}
\begin{aligned}
& \text{Suppose \eqref{kHessian} has a (continuous) viscosity solution. If } u \in USC\left(\overline{\Omega}\right) \text{ is a}\\
& \text{subsolution and } v \in LSC\left(\overline{\Omega}\right) \text{ is a supersolution of \eqref{kHessian}, then } u \leq v \text{ on } \overline{\Omega}.
\end{aligned}
\end{align}
The proof of this result is one of the main technical arguments in the viscosity solutions theory \cite{CIL}.

We remark that Definition~\ref{def:visc} allows for discontinuous viscosity solutions.  However, the comparison principle~(\ref{CP}) does not hold in this setting.  The theoretical details of discontinuous viscosity solutions are not well established, and are well beyond the scope of the present article.


\subsubsection*{$2$-Hessian equation}

In this paper, we focus on the the three-dimensional case with $k=2$
\[S_2[u] = f\]
where
\bq\label{sig2defn}
S_2[u] = \sig(\lam) = \lam_1 \lam_2 + \lam_1 \lam_3 + \lam_2 \lam_3.
\eq
The Dirichlet problem given by
\bq\label{2Hessian}\tag{2H}\begin{cases}
S_2[u] = f,& \quad \text{in } \Omega,\\
u = g, &\quad \text{on } \partial \Omega.
\end{cases}
\eq

\subsubsection*{Alternative description of $\Gamma_2$}

We have
\[
\Gamma_2 = \left\{\lambda \in \R^3 \mid \lam_1+\lam_2+\lam_3 > 0,~ \sig(\lam) > 0\right\}.
\]
The following Proposition provides an alternative characterization of $\Gamma_2$.
\begin{proposition}\label{prop2}
Let
\bq\label{gammadefn}
\Gamma = 
\left\{   
\lam \in \R^3  \mid \lam_1 + \lam_2 > 0, ~\lam_1 + \lam_3 > 0,~ \lam_2 + \lam_3 > 0
\right \}
\eq
Then
\[\Gamma_2  = \Gamma \cap \{\lam \in \R^3 \mid \sig(\lambda) > 0\}.\]
\end{proposition}
\begin{proof}
Proving the $\supseteq$ part is straightforward. We then prove the inclusion $\subseteq$. Suppose that $(\lam_1,\lam_2,\lam_3) \in \Gamma_2$. Without loss of generality we can assume that $\lam_1 \leq \lam_2 \leq \lam_3$. Thus, it is sufficient to show that $\lam_1+\lam_2 > 0$. Suppose that $\lam_1+\lam_2 \leq 0$. We consider two cases, each leading to a contradiction.
\begin{itemize}
	\item $\lam_1+\lam_2 = 0$
\end{itemize}

We have $\lam_1 \lam_2 \leq 0$. Hence
\begin{align*}
\sig(\lam)	& = \lam_1 \lam_2 + \lam_1 \lam_3 + \lam_2 \lam_3\\
		& = \lam_1 \lam_2 + (\lam_1 + \lam_2) \lam_3\\
		& = \lam_1 \lam_2\\
		& \leq 0,
\end{align*}
contradicting the assumption $\sig(\lam) > 0$.
\begin{itemize}
	\item $\lam_1+\lam_2 < 0$
\end{itemize}

Since $\lam_1 \leq \lam_2$, we have $\lam_1 < 0$. Moreover
\[\sig(\lam) > 0 \Longleftrightarrow \lam_3(\lam_1+\lam_2) > -\lam_1\lam_2 \Longleftrightarrow \lam_3 < -\frac{\lam_1\lam_2}{\lam_1+\lam_2}\]
and
\[\lam_1+\lam_2+\lam_3 > 0 \Longleftrightarrow \lam_3 > -\lam_1-\lam_2\]
From the above two inequalities we get
\[-\lam_1-\lam_2 < -\frac{\lam_1\lam_2}{\lam_1+\lam_2}\]
which we can rewrite as
\[\lam_1 (\lam_1 + \lam_2) + \lam_2^2 < 0.\]
Now, since $\lam_1 < 0$ and $\lam_1 + \lam_2 < 0$, the left-end side of the inequality must be positive and we have thus derived a contradiction.
\end{proof}

It is easy to show, using differentiation, that the function $\sig$ is nondecreasing on the set $\Gamma$, which gives some insight to why the set of admissible functions is the set of functions where $S_2$ is elliptic.

The constraint $\sig(\lam) \geq 0$ will be enforced automatically in our schemes by taking a non-negative $f$ in the PDE \eqref{2Hessian}. Therefore it is sufficient to look at the set $\Gamma$ as defined in \eqref{gammadefn}. We will refer to this restriction as \emph{plane-subharmonic} since it corresponds to $u$ being subharmonic on every plane.

\subsubsection*{Alternative description  of the $2$-Hessian operator}\label{sec:2hessian}
 
For a $3\times 3$ matrix $M$, the characteristic polynomial is given by
\[
\det (M) - c(M) \lam +  \trace (M)\lam^2 - \lam^3
\]
where $c(M)$, the sum of the principal minors of $M$, is given by
\bq\label{DiagM}
c(M) = \frac 1 2  \left ( \trace(M)^2  - \trace(M^2) \right ).
\eq
If $\lam_1$, $\lam_2$ and $\lam_3$ are the eigenvalues of $M$ then
\[c(M) =  \lam_1 \lam_2 + \lam_1 \lam_3 + \lam_2 \lam_3.\]
Therefore, using \eqref{sig2defn}, we conclude that \eqref{2Hessiandef} holds,
\[
S_2[u] = c\left(D^2u\right) = u_{xx}u_{yy}+u_{xx}u_{zz}+u_{yy}u_{zz}-u_{xy}^2-u_{xz}^2-u_{yz}^2.
\]

\subsubsection*{Linearization}

The linearization of $c(M)$ defined in \eqref{DiagM}, is given by:
\[\nabla c(M) \cdot N = \trace(M)\trace(N) - \trace(MN).\]

We can apply the linearization of $c(M)$ to obtain the linearization of the $2$-Hessian operator, $S_2[u]$, for $u\in C^2$, 
\bq\label{lin2Hessian}
\nabla S_2[u]\cdot\nu = \trace(D^2u)\trace(D^2\nu) - \trace(D^2uD^2\nu).
\eq

\begin{lemma}
Let $u \in C^2$. The linearization of the $2-$Hessian operator \eqref{lin2Hessian} is elliptic if $u$ is $2$-admissible.
\end{lemma}
\begin{proof}
Without loss of generality, we choose coordinates such that $D^2u(x)$ is diagonal. We can then rewrite the linearization of the $2$-Hessian operator as
\[\nabla S_2[u]\cdot\nu = \trace(A D^2\nu)\]
where $A = \diag(\lambda_2+\lambda_3,\lambda_1+\lambda_3,\lambda_1+\lambda_2)$. Hence, the linearization is elliptic if $A$ is positive definite, which is true if $u$ is $2$-admissible.
(It also follows directly from the definition of nonlinear elliptic operator (in the sense of \cite{CIL}) that the linearization is elliptic.)
\end{proof}

\begin{remark}
When the function $u$ fails to be ``strictly'' $2$-admissible, the linearization can be degenerate elliptic, which affects the conditioning of the linear system \eqref{lin2Hessian}. When $u$ is not $2$-admissible, the linear system can be unstable.
\end{remark}

\section{Discretization and solvers} 
In this section we explain why the naive finite difference method fails in general. We introduce explicit, semi-implicit, and Newton solvers for the naive finite difference method, which perform better by enforcing the plane-subharmonic constraint. This is similar to the solvers used in~\cite{BenamouFroeseObermanMA} for the Monge-Amp\`{e}re equation. Then we introduce a discretization which is monotone and thus provably convergent.

While the monotone discretization is less accurate, it has the advantage that it gives a globally consistent, monotone  discretization of the operator, meaning that we can apply the operator to non-admissible functions.   This is useful because it circumvents the need for special initial data, and allows for the parabolic (time-dependent) operator to be defined on an unconstrained class of functions.

In addition, we could combine the monotone discretization with the naive finite difference discretization to obtain provably convergent, accurate filtered finite difference schemes, using the ideas in \cite{FroeseObermanFiltered}.  This approach combines the advantages of both schemes, with little additional effort.   In this work, we were mainly interested in comparing the performance of the two schemes, so we did not implement the filtered scheme.

\subsection{Naive finite difference scheme}

We begin by discussing the naive finite difference discretization of the $2$-Hessian. This is done by simply using standard finite differences to discretize the operator. Denote by $D^{2,h}u$ the discretized Hessian using standard finite differences on a uniform grid with grid spacing $h$, i.e.,
\[
D^{2,h} u_{ijk} = \left[\begin{array}{ccc}
\Dxx u_{ijk} & \Dxy u_{ijk} & \Dxz u_{ijk} \\
\Dxy u_{ijk} & \Dyy u_{ijk} & \Dyz u_{ijk} \\
\Dxz u_{ijk} & \Dyz u_{ijk} & \Dzz u_{ijk} \end{array}\right],
\]
where, e.g.,
\begin{align*}
\Dt_{xx}u_{ijk} &= \frac{u_{i,j+1,k}-2u_{i,j,k}+u_{i,j-1,k}}{h^2},\\
\Dt_{xy}u_{ijk} &= \frac{u_{i+1,j+1,k}+u_{i-1,j-1,k}-u_{i-1,j+1,k}-u_{i+1,j-1,k}}{4h^2}.
\end{align*}

We then get the discrete version of the 2-Hessian operator $S_2[u]$ as
\bq\label{naivediscrete2Hessian}
S_2^A[u] = c\left(D^{2,h}u\right)
\eq
Since we are using centered finite differences,  this discretization is consistent, and it is second order accurate if the solution is smooth (hence the superscript $A$). However, this scheme is not monotone due to the off-diagonal terms in the cross derivatives $u_{xy}$, $u_{xz}$ and $u_{yz}$. Therefore the Barles and Souganidis theory~\cite{BSnum} does not apply and no convergence proof is available.

\subsection{Failure of the parabolic solver for the naive finite differences}
In this section we give a simple example to illustrate that the use of the naive finite difference scheme \eqref{naivediscrete2Hessian} together with a parabolic solver fails to converge.

The parabolic solver is given by
\bq\label{naiveparabolic}
u^{n+1} = u^n+dt(S_2^A[u]-f).
\eq
Consider the solution of \eqref{2Hessian} in $[0,1]^3$, given by
\[
u(\x) = \frac{\x^2}{2}, \quad f(\x) = 3.
\]
The iteration is initialized with the exact solution with noise from a uniform distribution $U(-0.01,0.01)$. The result after performing two iterations with the parabolic solver \eqref{naiveparabolic} with time step $dt = dx^4$ and the initial guess are illustrated in Figure~\ref{fig:Failure}. Regardless of the time step choosen ($dt = dx^4/10$ and $dt = dx^4/100$ were also used), after a sufficient number of iterations the solution behaves like in the example of Figure~\ref{fig:Failure}, until it eventually blows up. This tells us that the instability of the parabolic solver is inherent from the discretization rather than being the result of a poorly chosen time step. This instability is due to the fact that there is no mechanism to pick the right solution. The discretization, being a quadratic equation as we will see in ~\autoref{sec:jacobisolver}, has two solutions: the $2$-admissible solution we are looking for and the negative of this.

\begin{figure}[htp]
\centering
\begin{tabular}{cc}
\hspace{-1in}\includegraphics[width=0.65\textwidth]{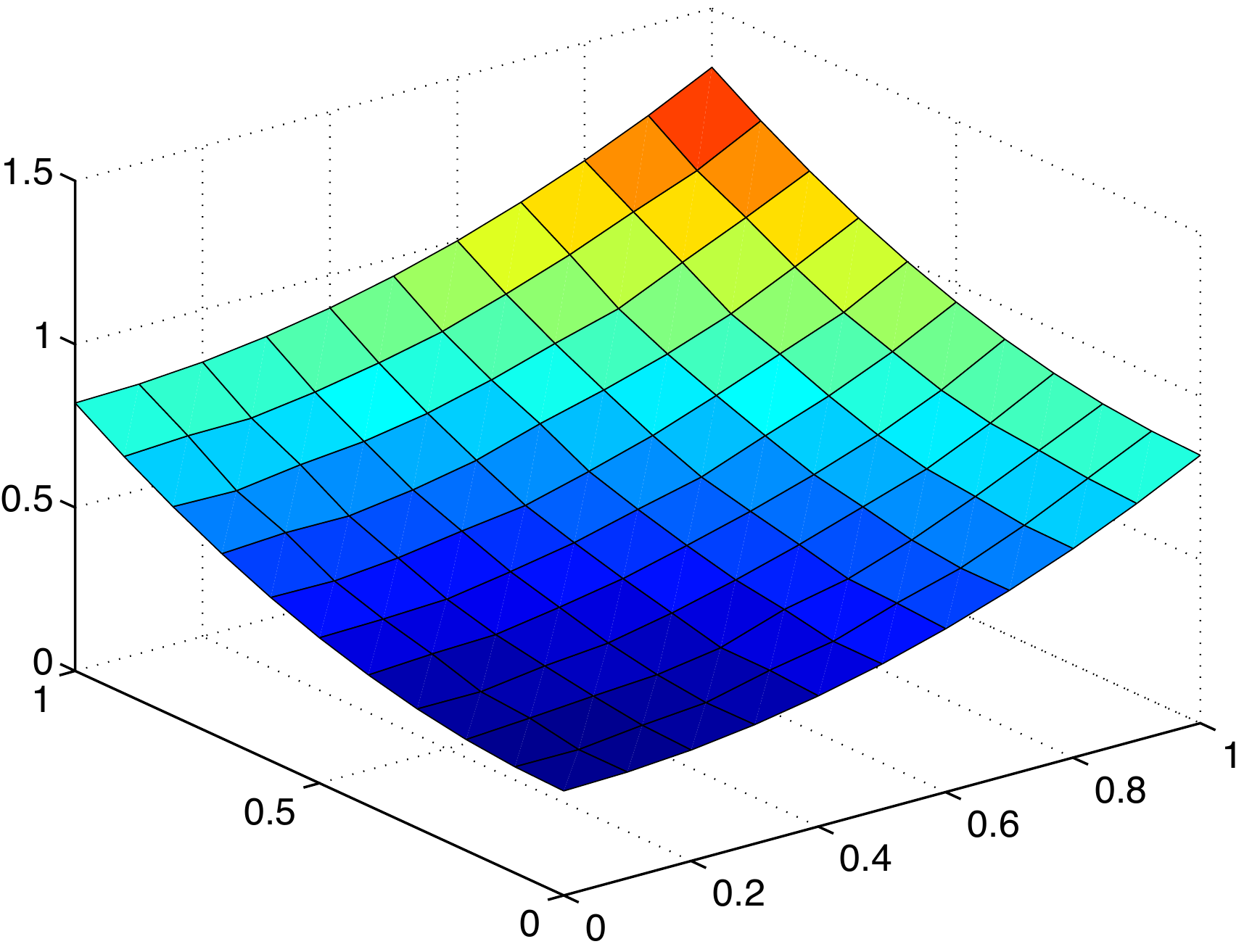} & \includegraphics[width=0.65\textwidth]{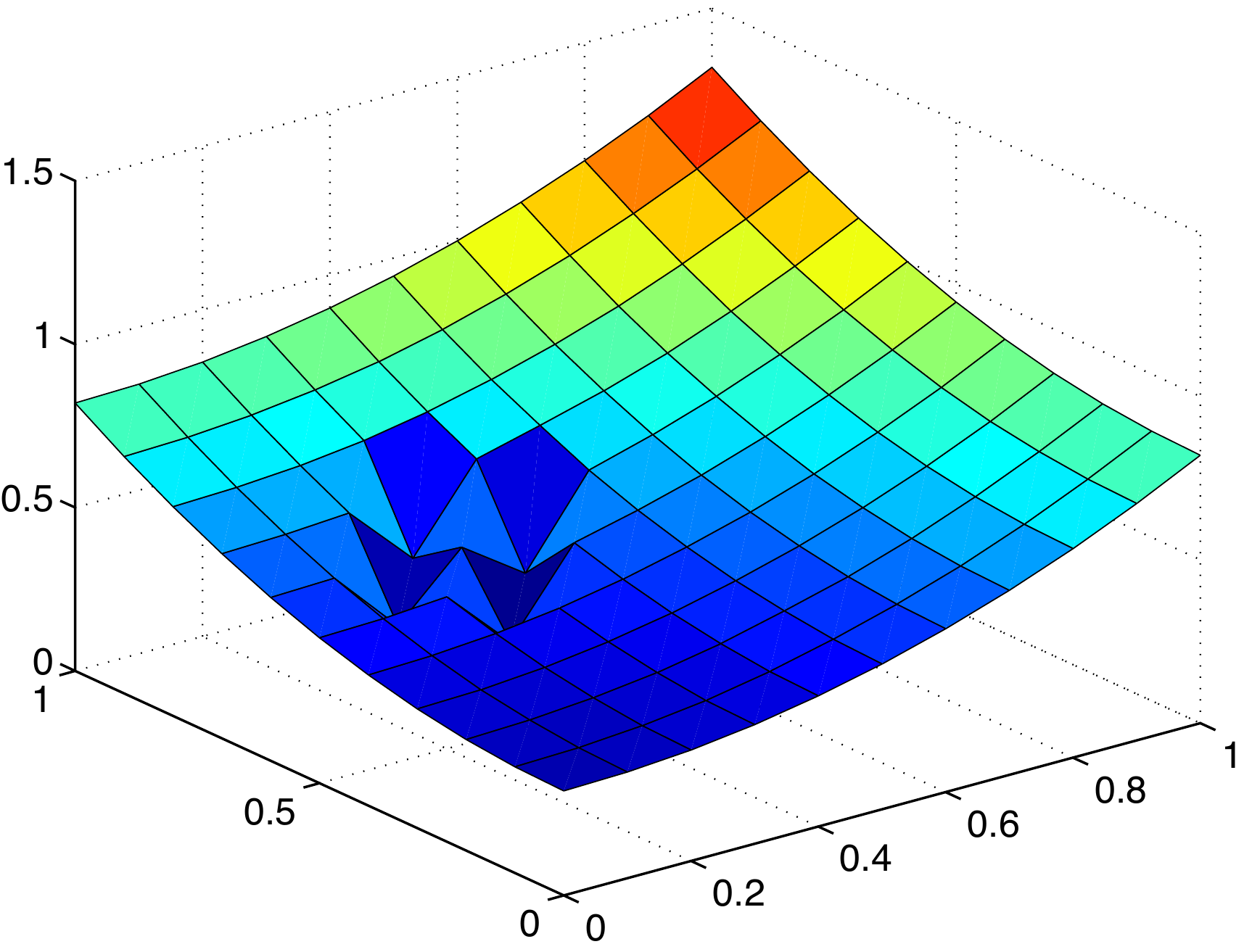}
\end{tabular}
\caption{Failure of the parabolic solver using the naive finite difference scheme: section $z = 0.9$ of the initial guess (left) and the solution after 25 iterations (right).}
\label{fig:Failure}
\end{figure}

\subsection{Solvers for the naive finite difference scheme}
In this section we present three different solvers for the naive finite difference scheme: a Jacobi type solver obtained by solving the discretization for the reference variable; a semi-implicit solver based on an identity that relates the Laplacian and the $2$-Hessian operator; a Newton solver.

\subsubsection{Jacobi solver}\label{sec:jacobisolver}

The accurate discretization of \eqref{2Hessian} leads to a quadratic equation for the reference variable at each grid point. To see this we introduce the notation
\bq
\label{def_d}
\begin{aligned}
a_1& =\frac{u_{i+1,j,k}+u_{i-1,j,k}}{2} & a_2 &=\frac{u_{i,j+1,k}+u_{i,j-1,k}}{2} & a_3&=\frac{u_{i,j,k+1}+u_{i,j,k-1}}{2}\\
a_4& =\frac{u_{i+1,j+1,k}+u_{i-1,j-1,k}}{2} & a_5&=\frac{u_{i-1,j+1,k}+u_{i+1,j-1,k}}{2} & a_6& =\frac{u_{i+1,j,k+1}+u_{i-1,j,k-1}}{2}\\ 
a_7& =\frac{u_{i-1,j,k+1}+u_{i+1,j,k-1}}{2} & a_8&=\frac{u_{i,j+1,k+1}+u_{i,j-1,k-1}}{2}& a_9&=\frac{u_{i,j+1,k-1}+u_{i,j-1,k+1}}{2}
\end{aligned}
\eq
Using \eqref{naivediscrete2Hessian}, $S_2^A[u] = f$ can be rewritten as
\[
\frac{4}{h^4}\left(\sum_{i_1<i_2\leq 3}(a_{i_1}-u_{ijk})(a_{i_2}-u_{ijk})\right) = f_{ijk} + \frac{1}{4h^4}\sum_{p=2}^4(a_{2p}-a_{2p+1})^2
\]

Solving for $u_{ijk}$  and selecting the smaller root (in order to select the locally more plane-subharmonic solution), we obtain
\bq
\label{J}
\tag{J}
u_{ijk} = \frac{a_1+a_2+a_3}{3}-\frac{1}{12}\sqrt{8\sum_{i_1<i_2\leq 3}(a_{i_1}-a_{i_2})^2+3\sum_{p=2}^4(a_{2p}-a_{2p+1})^2+12f_{ijk}h^4}.
\eq

We can now use a Jacobi iteration to find the fixed point of~\eqref{J}. Notice that the plane-subharmonic constraint is not enforced beyond the selection of the smaller root in~\eqref{J}.

\begin{remark}
Formula~\eqref{J} can be rewritten as
\[
u_{ijk} = \frac{a_1+a_2+a_3}{3}-\frac{h^2}{6}\sqrt{\trace(D^{2,h}u_{ijk})^2+3\left(f_{ijk}-S^{2,A}_h[u]\right)}.
\]
\end{remark}

\begin{remark}
Formula~\eqref{J} can also be used in a Gauss-Seidel iteration, which should converge faster than the Jacobi iteration. We choose not to implement it here since all computational results were obtained in MATLAB, which is known to be slow with loops.
\end{remark}

In order to prove the convergence of the above solver, is is sufficient to prove that it is monotone, which in this case is the same as showing that the value $u_{ijk}$ is a non-decreasing function of the neighboring values \cite{ObermanSINUM}. However, this is not the case for \eqref{J}.

\subsubsection{Semi-implicit solver}

The next solver we discuss is a semi-implicit one, which involves solving a Laplace equation at each iteration.

We begin with the following identity for the Laplacian in three dimensions:
\[|\Delta u| = \sqrt{(\Delta u)^2} = \sqrt{u_{xx}^2+u_{yy}^2+u_{zz}^2+2u_{xx}u_{yy}+2u_{xx}u_{zz}+2u_{yy}u_{zz}}.\]
If $u$ solves the $2$-Hessian equation, then
\[|\Delta u| = \sqrt{(\Delta u)^2} = \sqrt{u_{xx}^2+u_{yy}^2+u_{zz}^2+2u_{xy}^2+2u_{xz}^2+2u_{yz}^2+2f} = \sqrt{|D^2u|^2+2f}.\]
This leads to a semi-implicit scheme for solving the $2$-Hessian equation given by
\bq\label{semi}
\Delta u^{n+1} = \sqrt{|D^2u^n|^2+2f}.
\eq
Note that if $u$ is a 2-admissible function, then $\Delta u \geq 0$, a condition the scheme enforces.

A good initial value for the iteration is given by the solution of
\[\Delta u^0 = \sqrt{2f}.\]

\subsubsection{Newton solver}

To solve the discretized equation
\[S_2^A[u]=f\]
we can also use a damped Newton iteration
\[u^{n+1} = u^n-\alpha v^n\]
where $0 < \alpha \leq 1$. The damping parameter $\alpha$ is chosen at each step to ensure that the residual $\Norm{S_2^A[u^n]-f}$ is decreasing. (In practice we can often take $\alpha = 1$, but damping is sometimes needed.) The corrector $v^n$ solves the linear system
\[\left(\nabla_u S_2^A[u^n]\right)v^n = S_2^A[u^n]-f.\]
To setup the above equation we need the Jacobian of the scheme, which is given by
\[\nabla_u S_2^A[u] = \sum_{\substack{\nu_1,\nu_2 \in \{x,y,z\},\\\nu_1 \neq \nu_2}}(\Dt_{\nu_1\nu_1}u)\Dt_{\nu_2\nu_2}-(\Dt_{\nu_1\nu_2}u)\Dt_{\nu_1\nu_2}\]
Notice that it corresponds to the discrete version of the linearization of the $2$-Hessian equation \eqref{lin2Hessian}.

\subsection{Monotone finite difference scheme}
In this section we construct a monotone finite difference scheme. As we saw before, the naive approach of simply using standard finite differences for the terms in the Hessian matrix will not work because the cross derivative terms $u_{xy}$, $u_{xz}$ and $u_{yz}$ are not monotone. Instead the idea is to use wide stencils and a rotated coordinate system in which the Hessian matrix is diagonal. However, this coordinate system must be found in a monotone way. This section is divided in four parts: first, we briefly recall why it is enough to prove that our scheme is consistent and degenerate elliptic (and thus monotone) to conclude that it is convergent; second, we extend the function $\sig$ \eqref{sig2defn} to be non-decreasing in $\R^3$; third, we find an expression for the $2$-Hesssian operator $S_2[u]$ which can be discretized in a monotone manner; and fourth, we present the monotone finite difference scheme.



\subsubsection{Convergence of consistent degenerate elliptic scheme}
The convergence of our finite difference schemes relies, as usual, on the framework developed by Barles and Souganidis \cite{BSnum} and its extension in \cite{ObermanSINUM}.

The framework in \cite{BSnum} provides us with sufficient conditions for the convergence of approximation schemes to the unique viscosity solution of a PDE.

\begin{theorem}
Consider an elliptic equation that satisfies a comparison principle. A consistent, stable and monotone approximation scheme converges locally uniformly to the (unique) viscosity solution.
\end{theorem}

This framework, however, does not provided a method to verify monotonicity and stability. The work in \cite{ObermanSINUM} accomplishes precisely that.

Our finite difference schemes have the form
\[F[u] = F(u_i,u_{j \in N(i)}-u_i)\]
where $N(i)$ is the list of neighbours of $u_i$. We say that $F$ is degenerate elliptic if $F$ is nondecreasing in each variable.

The following Theorem, which can be found in \cite{ObermanSINUM}, yields a simple condition to verify both monotonicity and stability.

\begin{theorem}
A scheme is monotone and nonexpansive in the $l^\infty$ norm if and only if it is degenerate elliptic.
\end{theorem}

Consequently, proving that a scheme is convergent is reduced to checking two conditions: consistency and degenerate ellipticity.

\subsubsection{Non-decreasing extension of the operator}
In this section we find a non-decreasing extension of $\sig$ from $\Gamma$ to $\R^3$. Our ultimate goal is to build a monotone finite difference approximation of the $2-$Hessian equation.  Since  we know that the eigenvalues of admissible solutions $u$ belong to the set $\Gamma$, we are free to redefine $\sig$ outside of $\Gamma$ in order to ensure convergence. We then require an extension of $\sig$ that is non-decreasing in $\R^3$, which is accomplished in the following Lemma.

\begin{lemma}\label{lemma:monext}
The function $\sigex = f \circ \sort$ where $\sort$ denotes the sorting function and $f$ is given by
\[
f(x,y,z) = x\max(y,\abs{x}) + x\max(z,\abs{x}) +\max(y,\abs{x}) \max(z,\abs{x})
\]
extends $\sig$ on $\Gamma$ and is non-decreasing in $\R^3$.
\end{lemma}
\begin{proof}
Without loss of generality, we assume that $x \le y \le z$ since sorting the values is monotone. Moreover, we can rewrite $f$ as
\[
f(x,y,z) = \max\left (y+x,|x| + x\right ) \max\left (z+x,\abs{x} + x\right )- x^2.
\]

Suppose $(x,y,z) \in \Gamma$, then we recover $\sig(x,y,z)$.

Next we show that $\sigex$ is non-decreasing as a function of $(x,y,z)$. We have two cases to consider:
\begin{itemize}
	\item $x+y \geq 0$
\end{itemize}
Since $x \leq y \leq z$, $(x,y,z) \in \Gamma$ and so we recover $\sig$ which we know to be a non-decreasing function in $\Gamma$.

\begin{itemize}
	\item $x+y < 0$
\end{itemize}
Since $x \leq y \leq z$, $x < 0$. We then get $\sigex(x,y,z) = -x^2$, which is increasing since $x < 0$.


Hence $\sigex$ is non-decreasing.
\end{proof}

\subsubsection{Elliptic expression for the operator}
In this section we build an expression that can be discretized in a monotone way.

The idea is to mimic what was done for the Monge-Amp\`{e}re equation in \cite{ObermanFroeseFast}: use a matrix identity to obtain a monotone expression for the operator.

First note that $\trace(M)$ is invariant over conjugation $O^T M O$ by orthogonal matrices $O$. Second note that $\trace(M^2) = \sum_{ij} m_{ij}^2 \ge \sum_i m_{ii}^2$ with equality when $M$ is diagonal. Hence we have 
\[
\trace(M)^2 - \trace(M^2) \le \trace(O^TMO)^2 - \sum_i (O^TMO)_{ii}^2
\]
and therefore
\[
2 c(M) = \min_{\substack{O^TO = I,\\ R = O^TMO}} \left\{\left(\sum_i r_{ii}\right )^2 - \sum_i r_{ii}^2\right\},
\]
which can be rewritten as
\bq
c(M) = \min_{\substack{O^TO = I,\\ R = O^TMO}}\sig(\diag(R)),
\eq
where $\diag(R) =(r_{11}, r_{22}, r_{33})$ is the vector which is the diagonal of the matrix $R$  and $\sig$ is defined by~\eqref{sig2defn}. Thus, we have just proved the following Lemma.

\begin{lemma}\label{variational2Hessian}
Let $M$ be a $3\times 3$ symmetric matrix and $V$ be the set of all orthonormal bases of $\R^3$:
\[V = \left\{(\nu_1,\nu_2,\nu_3) \mid \nu_i \in \R^3, \nu_i~\bot~\nu_j \text{ if } i \neq j, \norm{\nu_i}_2 = 1 \right\}.\]
Then
\bq
c(M) = \min_{(\nu_1,\nu_2,\nu_3)\in V}\sig\left(\nu_1^TM\nu_1,\nu_2^TM\nu_2,\nu_3^TM\nu_3\right).
\eq
\end{lemma}

We can now use Lemma \ref{variational2Hessian} to characterize the $2$-Hessian operator of a $C^2$ function by expressing it in terms of second directional derivatives of $u$ as follows:
\bq\label{monotoneidentity}
S_2[u] = \min_{(\nu_1,\nu_2,\nu_3)\in V}\sig\left(\frac{\partial^2 u}{\partial \nu_1^2},\frac{\partial^2 u}{\partial \nu_2^2},\frac{\partial^2 u}{\partial \nu_3^2}\right).
\eq

\subsubsection{Monotone operator}
We now present the monotone discretization of the $2$-Hessian operator.

We approximate the second derivatives using centered finite differences which leads to a \emph{spatial} discretization with parameter $h$. In addition, we consider a finite number of possible directions $\nu$ that lie on the grid, thus introducing the \emph{directional} discretization with parameter $d\theta$. We denote the set of orthogonal basis available on the grid by $\mathcal{G}$. We then have
\bq\label{monotone2Hessian}\tag*{$(2H)^M$}
S_2^M[u] = \min_{(\nu_1,\nu_2,\nu_3)\in\mathcal{G}} \sigex\left(\Dt_{\nu_1\nu_1}u,\Dt_{\nu_2\nu_2}u,\Dt_{\nu_3\nu_3}u\right),
\eq
where $\mathcal{D}_{\nu\nu}$ is the finite difference operator for the second directional derivative in the direction $\nu$ which lies on the finite difference grid and are given by
\[\Dt_{\nu\nu}u(x_i) = \frac{1}{|\nu|^2h^2}(u(x_i+h\nu)+u(x_i-h\nu)-2u(x_i)).\]
Depending on the direction of the vector $\nu$, this may involve a wide stencil.

We define $d\theta$ as
\[d\theta = \max_{(w_1,w_2,w_3)\in V} \min_{(\nu_1,\nu_2,\nu_3)\in\mathcal{G}} \max\left\{\arccos\left(\frac{w_1^T\nu_1}{\norm{\nu_1}}\right),\arccos\left(\frac{w_2^T\nu_2}{\norm{\nu_2}}\right),\arccos\left(\frac{w_3^T\nu_3}{\norm{\nu_3}}\right)\right\}.\]

We now define $\mathcal{G}$ in more detail. Let $n_\theta$ denote the width of the stencil and set
\[\mathcal{V}_1 = \left\{\nu \in \Z^3: |\nu_i| \leq 1, \norm{\nu} \neq 0\right\}\]
and for $n_\theta \geq 2$
\[\mathcal{V}_{n_\theta} = \left\{\nu \in \Z^3: |\nu_i| \leq n_\theta, \forall_{|t|<1} \: t\nu \notin \mathcal{V}_{n_\theta-1} \right\}.\]
We then have
\[\mathcal{G}_{n_\theta} = \left\{(\nu_1,\nu_2,\nu_3) \in \mathcal{V}_{n_\theta}^3: \nu_i~\bot~\nu_j \text{ if } i \neq j\right\}.\]

We will refer to the monotone schemes with respect to the number of points in the stencil. For instance, the monotone scheme with the stencil of length $1$ (i.e., $n_\theta = 1$) has $n_S+1 = 27$ points.

\begin{remark}
Given that $\sigex$ is a symmetric function when implementing the monotone scheme we do not need to look into all the triplets in $\mathcal{G}_{n_\theta}$. For instance, for $n_\theta = 1$ we only need to look for the triplets in Table \ref{table:directionmonotonescheme}.
\end{remark}

\begin{table}[htp]
\centering\footnotesize
\begin{tabular}{c|c|c}
$v_1$		& $v_2$		& $v_3$ \\\hline
$(1,1,0)$		& $(1,-1,0)$	& $(0,0,1)$ \\
$(1,0,1)$		& $(1,0,-1)$	& $(0,1,0)$ \\
$(1,0,0)$		& $(0,1,1)$	& $(0,1,-1)$ \\
$(1,0,0)$		& $(0,1,0)$	& $(0,0,1)$
\end{tabular}
\caption{Elements of $\mathcal{G}_1$ up to permutations.}
\label{table:directionmonotonescheme}
\end{table}

\begin{table}[htp]
\centering\footnotesize
\begin{tabular}{c|c|c|c|c|c|c}
$n_\theta$	& $1$	& $2$	& $3$	& $4$	& $5$	& $6$ \\ \hline
$n_S$		& $26$	& $98$	& $290$	& $579$	& $1155$	& $1731$
\end{tabular}
\caption{$n_S$ is the number of $\nu$ directions available in the stencil, i.e., $n_S = \# \mathcal{V}_{n_\theta}$}
\label{table:monotonescheme}
\end{table}

\begin{figure}[htp]
\centering
\begin{tabular}{cc}
\includegraphics[width=0.5\textwidth]{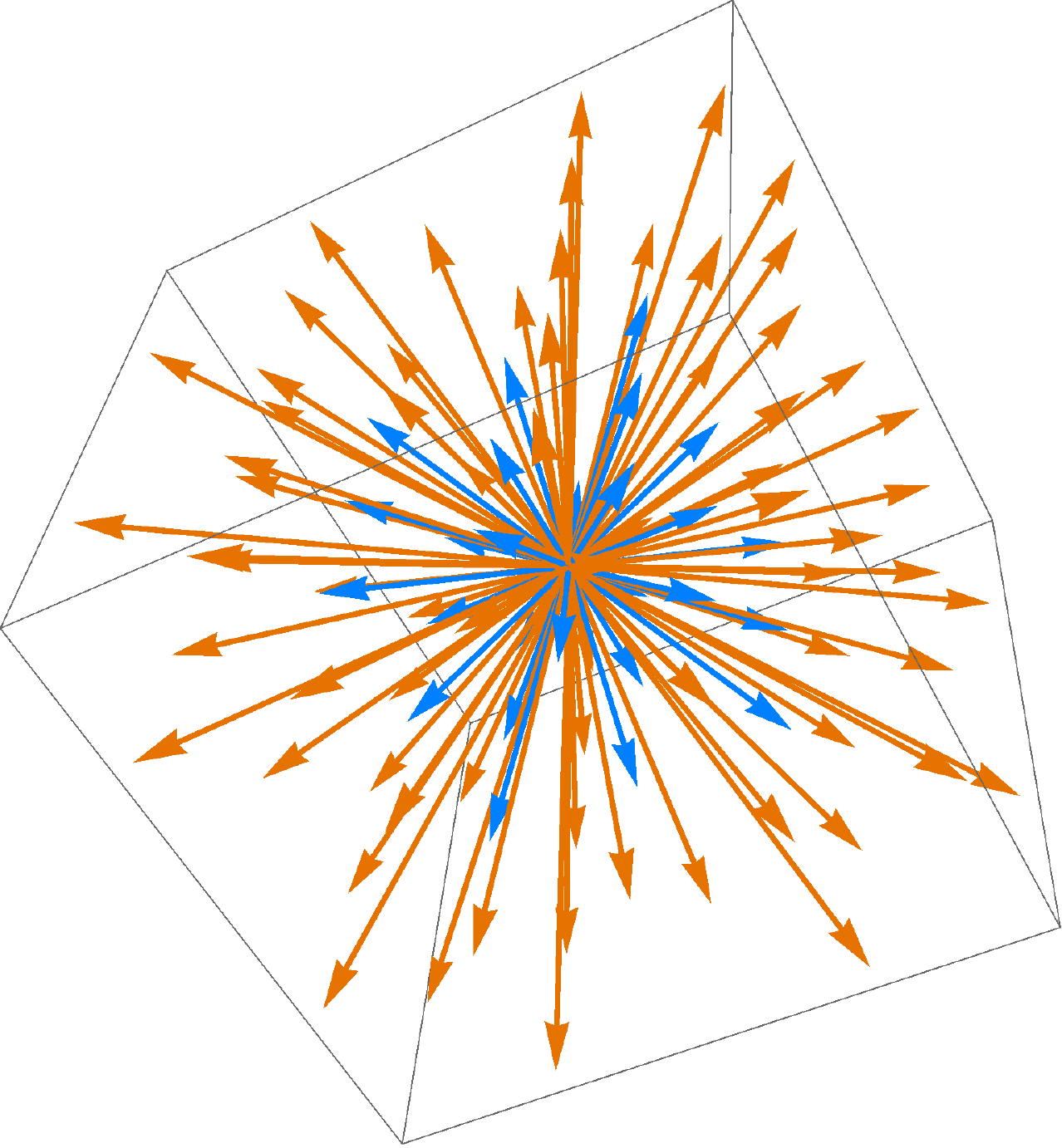}\end{tabular}
\caption{Elements of $\mathcal{V}_1$ (blue) and elements of $\mathcal{V}_2 \setminus \mathcal{V}_1$ (orange).}
\label{fig:Ex7curve}
\end{figure}

We now give the proof of the convergence of the monotone scheme. In order to do that, we first need to define our scheme at the boundary. Since we choose our domain to be the box $[0,1]^3$, the grid points are aligned with the boundary and so we simply have to set $g$ at those nodes. Set
\bq\tag{M}\label{monotone}F^M[u](x) = \begin{cases}
S_2^M[u](x) - f(x),	& \text{if } x \in \Omega,\\
u(x)-g(x),		& \text{if } x \in \partial \Omega.
\end{cases}
\eq

\begin{lemma}\label{degenerateelliptic}
The finite difference scheme given by \eqref{monotone} is degenerate elliptic.
\end{lemma}

\begin{proof}
From the definition, the discrete second directional derivatives $\Dt_{\nu\nu}$ are nondecreasing functions of the differences between neighboring values and reference values, $u_j-u_i$, where $u_j$ is one of the neighboring values of $u_i$ in the direction $\nu$. The scheme \ref{monotone2Hessian} is a nondecreasing combination of the operators $\min$ and $\sigex$ (the latter proved in Lemma \ref{lemma:monext} to be nondecreasing) applied to the degenerate elliptic terms $\Dt_{\nu\nu}$, and so it is also degenerate elliptic. It is also clear that $u-g$ is degenerate elliptic. Hence, \eqref{monotone} is degenerate elliptic.
\end{proof}

\begin{lemma}\label{consistency}
Let $x_0 \in \Omega$ be a reference point on the grid and $\phi$ be a $C^4$ function that is defined in a neighborhood of the grid. Then the scheme $S_2^M[\phi]$ defined in \ref{monotone2Hessian} approximates \eqref{2Hessian} with accuracy
\[S_2^M[\phi] = S_2[\phi] + \bO(h^2 + d\theta).\]
\end{lemma}

\begin{proof}
From a simple Taylor series computation we have
\[\Dt_{\nu\nu}\phi(x_0) = \frac{\partial^2 \phi}{\partial \nu^2}(x_0) + \bO(h^2).\]

Using~\eqref{monotoneidentity} we can rewrite the $2$-Hessian operator as
\[S_2[\phi] = \min_{(\nu_1,\nu_2,\nu_3)\in V}\sig\left(\frac{\partial^2 \phi}{\partial \nu_1^2},\frac{\partial^2 \phi}{\partial \nu_2^2},\frac{\partial^2 \phi}{\partial \nu_3^2}\right) = \sig\left(\frac{\partial^2 \phi}{\partial v_1^2},\frac{\partial^2 \phi}{\partial v_2^2},\frac{\partial^2 \phi}{\partial v_3^2}\right),\]
where the $v_j$ are orthogonal unit vectors, which may not be in the set of grid vectors $\mathcal{G}$. We know that by definition of $d\theta$ we have
\[\min_{(\nu_1,\nu_2,\nu_3)\in\mathcal{G}} \max\left\{\arccos\left(\frac{v_1^T\nu_1}{\norm{\nu_1}}\right),\arccos\left(\frac{v_2^T\nu_2}{\norm{\nu_2}}\right),\arccos\left(\frac{v_3^T\nu_3}{\norm{\nu_3}}\right)\right\} \leq d\theta.\]
Let then $w\in \mathcal{G}$ where the above min is attained. Then the angle between between each $v_j$ and $w_j$ is less or equal than $d\theta$ and so there is $dv_j$ such that
\[v_j + dv_j = \frac{w_j}{\norm{w_j}}\]
with $\norm{dv_j} = \bO(d\theta)$.

Now we consider the discretized problem
\begin{align*}
S_2^M[\phi]	& = \min_{(\nu_1,\nu_2,\nu_3)\in \mathcal{G}}\sig\left(\Dt_{\nu_1\nu_1}\phi,\Dt_{\nu_2\nu_2}\phi,\Dt_{\nu_3\nu_3}\phi\right)\\
			& \leq \sig\left(\Dt_{w_1w_1}\phi,\Dt_{w_2w_2}\phi,\Dt_{w_3w_3}\phi\right)\\
			& = \sig\left(\frac{\partial^2 \phi}{\partial w_1^2},\frac{\partial^2 \phi}{\partial w_2^2},\frac{\partial^2 \phi}{\partial w_3^2}\right) + \bO(h^2)\\
			& = \sig\left(\frac{\partial^2 \phi}{\partial v_1^2},\frac{\partial^2 \phi}{\partial v_2^2},\frac{\partial^2 \phi}{\partial v_3^2}\right) + \bO(h^2+d\theta)\\
			& = \min_{(\nu_1,\nu_2,\nu_3)\in V}\sig\left(\frac{\partial^2 \phi}{\partial \nu_1^2},\frac{\partial^2 \phi}{\partial \nu_2^2},\frac{\partial^2 \phi}{\partial \nu_3^2}\right) + \bO(h^2+d\theta),
\end{align*}
where we used the fact that
\[
\frac{\partial^2 \phi}{\partial w_j^2} = \frac{\partial^2 \phi}{\partial v_j^2} + \bO(d\theta).
\]

In addition, since the set of grid vectors $\mathcal{G}$ is a subset of the set of all orthogonal vectors $V$ up to scaling, we find that
\begin{align*}
\min_{(\nu_1,\nu_2,\nu_3)\in \mathcal{G}}\sig\left(\Dt_{\nu_1\nu_1}\phi,\Dt_{\nu_2\nu_2}\phi,\Dt_{\nu_3\nu_3}\phi\right)		& \geq \min_{(\nu_1,\nu_2,\nu_3)\in V}\sig\left(\Dt_{\nu_1\nu_1}\phi,\Dt_{\nu_2\nu_2}\phi,\Dt_{\nu_3\nu_3}\phi\right)\\
						& = \min_{(\nu_1,\nu_2,\nu_3)\in V}\sig\left(\frac{\partial^2 \phi}{\partial \nu_1^2},\frac{\partial^2 \phi}{\partial \nu_2^2},\frac{\partial^2 \phi}{\partial \nu_3^2}\right) + \bO(h^2).
\end{align*}

Combining the two inequalities deduced above, we conclude the proof.
\end{proof}

\begin{theorem}
Suppose \eqref{2Hessian} has a continuous viscosity solution. Let $u_{h,d\theta}$ denote the solutions of the scheme \eqref{monotone} and $u$ denote the unique viscosity solution of \eqref{2Hessian}. Then, as $h,d\theta,h/d\theta \to 0$, $u_{h,d\theta}$ converges locally uniformly to $u$.
\end{theorem}
\begin{proof}
The convergence follows from verifying consistency and degenerate ellipticity, as explained above, by the Barles and Souganidis theory \cite{BSnum}. This is accomplished in Lemmas \ref{degenerateelliptic} and \ref{consistency}. Notice that the PDE \eqref{2Hessian} has a comparison principle \eqref{CP} as pointed out in \autoref{sec:viscosity}.
\end{proof}

\begin{remark}
The assumption of the existence of a continuous viscosity solution is required for the comparison principle. This assumption is restrictive since the existence result of Theorem~\ref{wellposedness} requires smooth data, which is not the case for the examples considered here. 
In fact, continuous viscosity solutions can exist in much more general settings.  However, a precise well-posedness result for the (weak) Dirichlet problem is not presently available, and the highly technical details require significant additional work that is beyond the scope of the present article.
\end{remark}

\subsection{Solvers for the monotone finite difference scheme}
In this section we present two solvers for the monotone finite difference scheme.

\subsubsection{Parabolic solver}
Using the monotone discretization $S_2^M[u]$, the simplest solver for the $2$-Hessian equation is to use the fixed point method
\bq\label{parabolicmethod}
u^{n+1} = u^n-\alpha(S_2^M[u]-f)
\eq
which corresponds to the discrete version of the parabolic equation $u_t = -S_2[u]+f$ using a forward Euler step. The fixed point iteration will be a contraction in the maximum norm provided that we choose $\alpha$ small enough, as dictated by the nonlinear CFL condition \cite{ObermanSINUM}, which in this case means $\alpha = \mathcal{O}(h^4)$. This will make the solver very slow. However, since we extended $\sig$ to be degenerate elliptic in $\R^3$, this is a global solver, meaning that it will converge regardless of the initial guess we choose.

\subsubsection{Newton solver}
As with the standard finite difference scheme, one can also use a (damped) Newton solver. In this case the Jacobian for the monotone discretization is obtained by using Danskin's Theorem \cite{Bertsekas} and the product rule:
\[\nabla_u S_2^M[u] = \begin{cases}
-2(\Dt_{\nu^*_1\nu^*_1}u)\Dt_{\nu^*_1\nu^*_1}, & \text{if} \: \Dt_{\nu^*_1\nu^*_1}u + \Dt_{\nu^*_2\nu^*_2}u < 0,\\
\sum\limits_{\substack{\nu_1,\nu_2 \in \{\nu^*_1,\nu^*_2,\nu^*_3\},\\\nu_1 \neq \nu_2}}(\Dt_{\nu_1\nu_1}u)\Dt_{\nu_2\nu_2}, & \text{otherwise,} \\
\end{cases}\]
where $\nu^*_j$ are the directions active in the minimum in \ref{monotone2Hessian}, with $\Dt_{\nu^*_1\nu^*_1}u \leq \Dt_{\nu^*_2\nu^*_2}u \leq \Dt_{\nu^*_3\nu^*_3}u$. Unlike the previous solver, this is a local solver, meaning that we need a good initial guess in order to have convergence.

\section{Computational results}
In this section we summarize the results of a number of different examples using the solvers described in the previous section. These computations are performed on a $N\times N\times N$ grid on the cube $[0,1]^3$. Unless otherwise mentioned, all solvers were initialized with an initial guess provided by the explicit method \eqref{J}, which we iterate until $\left|S_2^A[u^n]-f\right| < 10^{-1}$. The initial guess for the explicit method \eqref{J} was the exact solution with some noise from a uniform distribution. As stopping criteria for the Newton solver we used $\left|S_2^H[u^n]-f\right| < 10^{-10}$ where $H \in \{A,M\}$.
Solutions were also computed using \eqref{J} and \eqref{semi} with very similar results to the ones provided by the Newton solver being obtained. For that reason, we choose not to display them here.

\begin{remark}
Notice that at points near the boundary of the domain, some values required by the wide stencil will not be available. For this reason and to simplify things, we set the exact solution at those points. However it is important to point out that we can use interpolation at the boundary to construct a (lower accuracy) stencil, thus avoiding the need to initialize with the exact solution.
\end{remark}

\begin{example}[Quadratic function]
We consider the case where $u$ is a non-convex (but $2$-admissible function) given by
\bq\label{ex1}
u(\x) = x_1^2-\frac{1}{2}x_2^2+2x_3^2, \quad f(\x) = 2.
\eq
with $\x = (x_1,x_2,x_3)$.
In Table \ref{table:Ex1errors}, we compare the results obtained using standard finite differences and the monotone schemes with different stencil sizes. For this example, we used the Newton solver for all schemes.

All methods provide machine accuracy which is expected since the standard finite differences are exact for quadratic functions and the monotone schemes computed the desired directional derivative.
\end{example}

\begin{table}[htp]
\centering\footnotesize
\begin{tabular}{ccccccccc}
\multicolumn{9}{c}{Errors and order, $1^{st}$ Example}\\
N & \multicolumn{2}{c}{Standard} & \multicolumn{2}{c}{Monotone (27-point)} & \multicolumn{2}{c}{Monotone (99-point)} & \multicolumn{2}{c}{Monotone (291-point)}\\
\hline
15 & \num{4.441e-16} & - & \num{4.441e-16} & - & \num{4.441e-16} & - & \num{4.441e-16} & - \\
20 & \num{4.441e-16} & -0.00 & \num{8.882e-16} & -2.27 & \num{8.882e-16} & -2.27 & \num{6.661e-16} & -1.33 \\
25 & \num{4.441e-16} & -0.00 & \num{8.882e-16} & -0.00 & \num{8.882e-16} & -0.00 & \num{8.882e-16} & -1.23 \\
30 & \num{4.441e-16} & -0.00 & \num{1.332e-15} & -2.14 & \num{8.882e-16} & -0.00 & \num{8.882e-16} & -0.00 \\
35 & \num{4.441e-16} & -0.00 & \num{1.332e-15} & -0.00 & \num{8.882e-16} & -0.00 & \num{1.110e-15} & -1.40 \\
\end{tabular}
\caption{Accuracy in the $l^\infty$ norm and order of convergence of the schemes for the first example using the Newton solver.}
\label{table:Ex1errors}
\end{table}

\begin{example}[smooth convex radial function]
We consider now the case where $u$ is given by
\bq\label{ex2}
u(\x) = \exp\left(\frac{\norm{\x-\x_0}^2}{2}\right), \quad f(\x) = (3+2\norm{\x-\x_0}^2)\exp(\norm{\x-\x_0}^2).
\eq

The maximum errors are given in Table \ref{table:Ex2errors}. As in the previous example we used the Newton solver for all schemes.

The standard finite differences provided second order convergence, which was expected since the solution is smooth. The monotone schemes provided only first order convergence (or close to it).
\end{example}

\begin{table}[htp]
\centering\footnotesize
\begin{tabular}{ccccccccc}
\multicolumn{9}{c}{Errors and order, $2^{nd}$ Example}\\
N & \multicolumn{2}{c}{Standard} & \multicolumn{2}{c}{Monotone (27-point)} & \multicolumn{2}{c}{Monotone (99-point)} & \multicolumn{2}{c}{Monotone (291-point)}\\
\hline
15 & \num{2.393e-04} & - & \num{3.472e-04} & - & \num{2.167e-04} & - & \num{1.302e-04} & - \\
20 & \num{1.298e-04} & 2.00 & \num{2.225e-04} & 1.46 & \num{1.518e-04} & 1.17 & \num{1.034e-04} & 0.75 \\
25 & \num{8.197e-05} & 1.97 & \num{1.650e-04} & 1.28 & \num{1.165e-04} & 1.13 & \num{8.552e-05} & 0.81 \\
30 & \num{5.607e-05} & 2.01 & \num{1.346e-04} & 1.08 & \num{9.357e-05} & 1.16 & \num{7.216e-05} & 0.90 \\
35 & \num{4.091e-05} & 1.98 & \num{1.259e-04} & 0.42 & \num{7.809e-05} & 1.14 & \num{6.247e-05} & 0.91 \\
\end{tabular}
\caption{Accuracy in the $l^\infty$ norm and order of convergence of the schemes for the second example using the Newton solver.}
\label{table:Ex2errors}
\end{table}

\begin{example}[smooth non-convex radial function]\label{example:Ex3}
We consider now the case where $u$ is given by
\bq\label{ex3}
u(\x) = \exp\left(2x_1^2-x_2^2+4x_3^2\right), \quad f(\x) = 8\left(1+12x_1^2+6x_2^2+16x_3^2\right)\exp\left(4x_1^2-2x_2^2+8x_3^2\right).
\eq

The maximum errors are given in Table \ref{table:Ex3errors}. Once again the solutions were computed with a Newton solver for all schemes.

The standard finite differences demonstrates again second order convergence. For the monotone schemes, the error tappers off with the grid size and we only see an error reduction by considering wider stencils. This tells us that the directional resolution error dominates the spatial resolution error. It is important to point out that this doesn't contradict our theoretical results since the only thing we proved was that we have convergence as both $h$ and $d\theta$ go to $0$, which we observe here.
\end{example}

\begin{table}[htp]
\centering\footnotesize
\begin{tabular}{ccccccccc}
\multicolumn{9}{c}{Errors and order, $3^{rd}$ Example}\\
N & \multicolumn{2}{c}{Standard} & \multicolumn{2}{c}{Monotone (27-point)} & \multicolumn{2}{c}{Monotone (99-point)} & \multicolumn{2}{c}{Monotone (291-point)}\\
\hline
15 & \num{3.028e-04} & - & \num{3.287e-02} & - & \num{1.110e-02} & - & \num{5.044e-03} & - \\
20 & \num{1.669e-04} & 1.95 & \num{3.312e-02} & -0.02 & \num{1.211e-02} & -0.29 & \num{5.617e-03} & -0.35 \\
25 & \num{1.052e-04} & 1.98 & \num{3.305e-02} & 0.01 & \num{1.260e-02} & -0.17 & \num{5.920e-03} & -0.22 \\
30 & \num{7.218e-05} & 1.99 & \num{3.311e-02} & -0.01 & \num{1.306e-02} & -0.19 & \num{6.396e-03} & -0.41 \\
35 & \num{5.262e-05} & 1.99 & \num{3.302e-02} & 0.02 & \num{1.339e-02} & -0.16 & \num{6.703e-03} & -0.29 \\
\end{tabular}
\caption{Accuracy in the $l^\infty$ norm and order of convergence of the schemes for the third example using the Newton solver.}
\label{table:Ex3errors}
\end{table}

\begin{example}[smooth non-convex radial function]
We consider another example of smooth radial function which is non convex but $2$-admissible:
\bq\label{ex4}
u(\x) = \log(2+\norm{\x}^2), \quad f(\x) = -\frac{4(-6+\norm{\x}^2)}{(2+\norm{\x}^2)^3}.
\eq

The maximum errors are given in Table \ref{table:Ex4errors}. Once again the solutions were computed with a Newton solver, regardless of the scheme.

As in the previous example, standard finite differences provide second order convergence and only with wider stencils we see a decrease in error with the grid size. Moreover, the monotone schemes with wider stencils also exhibit second order convergence (before it tappers off in the case of the $99$-point stencil).
\end{example}

\begin{table}[htp]
\centering\footnotesize
\begin{tabular}{ccccccccc}
\multicolumn{9}{c}{Errors and order, $4^{th}$ Example}\\
N & \multicolumn{2}{c}{Standard} & \multicolumn{2}{c}{Monotone (27-point)} & \multicolumn{2}{c}{Monotone (99-point)} & \multicolumn{2}{c}{Monotone (291-point)}\\
\hline
15 & \num{4.723e-05} & - & \num{1.664e-03} & - & \num{3.882e-04} & - & \num{4.909e-04} & - \\
20 & \num{2.564e-05} & 2.00 & \num{1.668e-03} & -0.01 & \num{1.787e-04} & 2.54 & \num{2.500e-04} & 2.21 \\
25 & \num{1.615e-05} & 1.98 & \num{1.674e-03} & -0.01 & \num{1.007e-04} & 2.46 & \num{1.462e-04} & 2.30 \\
30 & \num{1.111e-05} & 1.98 & \num{1.672e-03} & 0.01 & \num{8.617e-05} & 0.82 & \num{9.063e-05} & 2.53 \\
35 & \num{8.052e-06} & 2.02 & \num{1.670e-03} & 0.01 & \num{9.620e-05} & -0.69 & \num{6.506e-05} & 2.08 \\
\end{tabular}
\caption{Accuracy in the $l^\infty$ norm and order of convergence of the schemes for the fourth example using the Newton solver.}
\label{table:Ex4errors}
\end{table}

\begin{example}[non smooth convex function]
We consider now the case where $u$ is given by
\bq\label{ex5}
u(\x) = \frac{1}{2}\left((\norm{\x-\x_0}-0.2)^+\right)^2, \quad f(\x) = \left(3+\frac{1}{25\norm{\x-\x_0}^2}-\frac{4}{5\norm{\x-\x_0}} \right)\textbf{1}_{\left\{\norm{\x-\x_0} > 0.2\right\}}(\x).
\eq

The maximum errors are given in Table \ref{table:Ex5errors}. Due to its degenerate ellipticity, the monotone schemes required the use of the damped Newton solver.

Despite the lack of smoothness of the solution, the Newton solver with standard finite differences still converged. As for the monotone scheme, there was a significant increase in the number of iterations required: the wider the stencil, the more iterations required (around 10 times more iterations when compared to the Newton solver for the naive finite differences in the worst cases).

For the $291$-stencil, as in Example \ref{example:Ex3}, the error tapers off, indicating that the directional resolution error dominates the spatial error and, again, we still see the convergence as both $h$ and $d\theta$ go to $0$.

\end{example}

\begin{table}[htp]
\centering\footnotesize
\begin{tabular}{ccccccccc}
\multicolumn{9}{c}{Errors and order, $5^{th}$ Example}\\
N & \multicolumn{2}{c}{Standard} & \multicolumn{2}{c}{Monotone (27-point)} & \multicolumn{2}{c}{Monotone (99-point)} & \multicolumn{2}{c}{Monotone (291-point)}\\
\hline
15 & \num{7.580e-04} & - & \num{2.261e-03} & - & \num{7.707e-04} & - & \num{5.086e-04} & - \\
20 & \num{6.506e-04} & 0.50 & \num{2.329e-03} & -0.10 & \num{7.235e-04} & 0.21 & \num{1.924e-04} & 3.18 \\
25 & \num{3.353e-04} & 2.84 & \num{2.057e-03} & 0.53 & \num{5.871e-04} & 0.89 & \num{1.758e-04} & 0.39 \\
30 & \num{3.032e-04} & 0.53 & \num{2.156e-03} & -0.25 & \num{5.431e-04} & 0.41 & \num{2.197e-04} & -1.18 \\
35 & \num{2.129e-04} & 2.22 & \num{2.018e-03} & 0.42 & \num{5.159e-04} & 0.32 & \num{2.351e-04} & -0.43 \\
\end{tabular}
\caption{Accuracy in the $l^\infty$ norm and order of convergence of the schemes for the fifth example using the Newton solver.}
\label{table:Ex5errors}
\end{table}

\begin{example}[example with blow-up]
We considered as well the case
\bq\label{ex6}
u(\x) = -\sqrt{3-\norm{\x}^2}, \quad f(x) = -\frac{-9+\Norm{\x}^2}{(-3+\Norm{\x}^2)^2}.
\eq
Notice that $f$ is unbounded at the boundary point $(1,1,1)$ and $u$ will be singular at that point as well. Despite that the Newton solver still converged, but with a smaller rate of convergence  (approximately 0.3). It is important to observe that in the case of the Monge-Amp\`{e}re, the Newton solver failed to converge in the analogue example. This may be because the Monge-Amp\`ere equation is more strongly nonlinear than the 2-Hessian equation. The better accuracy of the wider monotone schemes is explained by the fact that the exact solution is prescribed at more grid points near the boundary of the (computational) domain, in particular, where $u$ is singular and $f$ is unbounded.
\end{example}

\begin{table}[htp]
\centering\footnotesize
\begin{tabular}{ccccccccc}
\multicolumn{9}{c}{Errors and order, $6^{th}$ Example}\\
N & \multicolumn{2}{c}{Standard} & \multicolumn{2}{c}{Monotone (27-point)} & \multicolumn{2}{c}{Monotone (99-point)} & \multicolumn{2}{c}{Monotone (291-point)}\\
\hline
15 & \num{1.104e-03} & - & \num{5.627e-03} & - & \num{5.600e-04} & - & \num{3.026e-04} & - \\
20 & \num{1.096e-03} & 0.02 & \num{5.224e-03} & 0.24 & \num{4.229e-04} & 0.92 & \num{2.628e-04} & 0.46 \\
25 & \num{1.054e-03} & 0.17 & \num{4.891e-03} & 0.28 & \num{3.454e-04} & 0.87 & \num{2.344e-04} & 0.49 \\
30 & \num{1.007e-03} & 0.24 & \num{4.698e-03} & 0.21 & \num{2.921e-04} & 0.88 & \num{2.102e-04} & 0.58 \\
35 & \num{9.621e-04} & 0.29 & \num{4.612e-03} & 0.12 & \num{2.538e-04} & 0.89 & \num{1.906e-04} & 0.62 \\
\end{tabular}
\caption{Accuracy in the $l^\infty$ norm and order of convergence of the schemes for the sixth example using the Newton solver.}
\label{table:Ex6errors}
\end{table}

\begin{example}\label{surfacelevelsetscube}
We consider as well the example with $f \equiv 1$ and $g \equiv 0$ Dirichlet boundary conditions. No exact solution is known. In Figure \ref{fig:Ex7surfacelevelsetplots}, we illustrate some of the surface plots of the level sets $u=c$ of the solution with the standard finite differences and monotone scheme with $c\in\{-0.01,-0.03,-0.07\}$. Note that the zero level set ($c=0$) is the boundary of the cube $[0,1]^3$ where the zero Dirichlet boundary conditions are prescribed. The surface plots become spheres as $c$ decreases, with $c=-.01$ being the only where there's a tangible difference between the two schemes, most likely due to the expected higher accuracy from the standard finite differences. In Figure \ref{fig:Ex7curve}, we plot the curve $u(t,t,t)$ with $t \in [0,1]$ and see that there's a small difference between the solutions from the standard finite differences and the monotone scheme.
\end{example}

\begin{figure}[htp]
\centering
\begin{tabular}{cc}
\hspace{-1in}\includegraphics[width=0.65\textwidth]{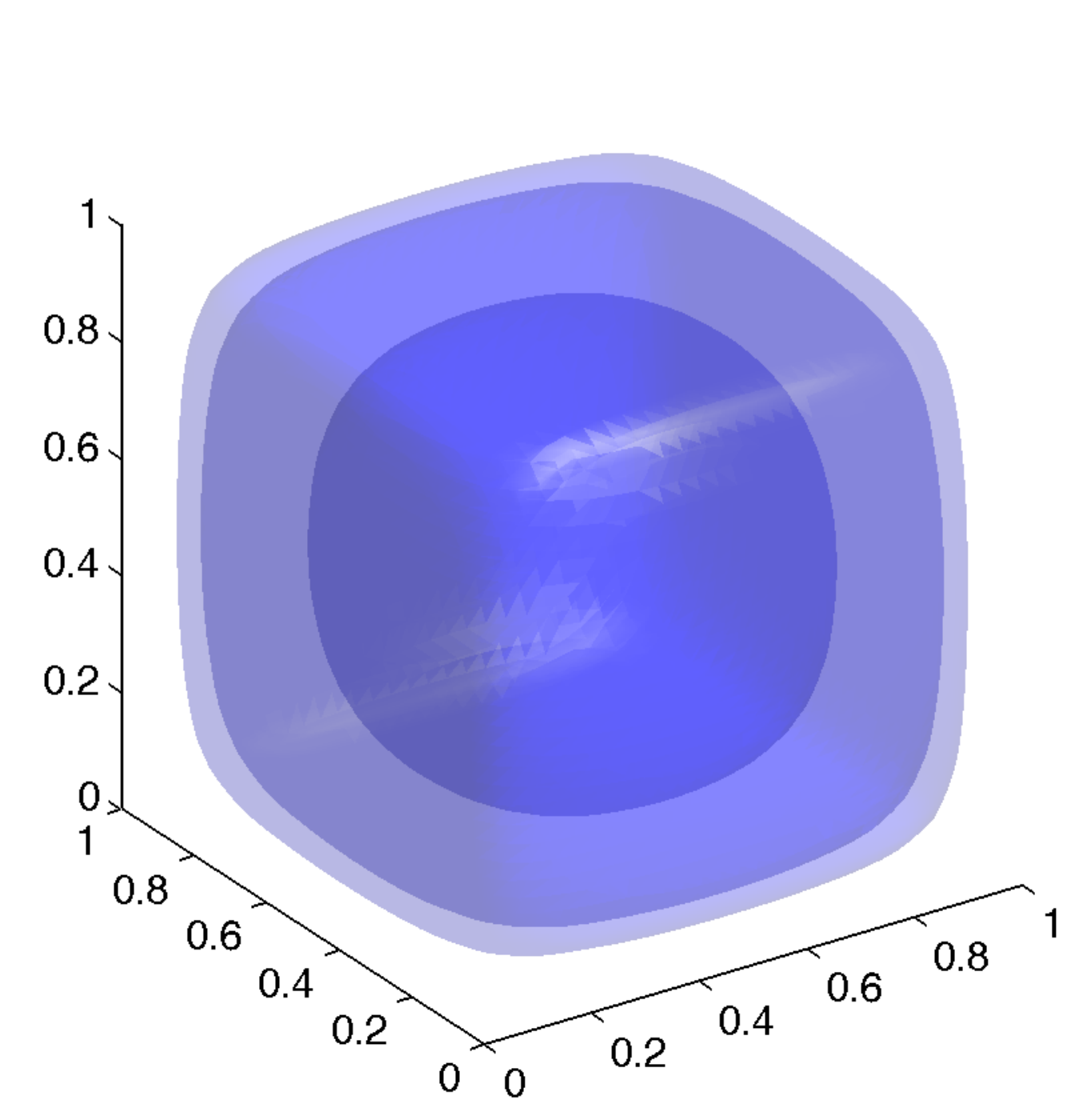} & \includegraphics[width=0.65\textwidth]{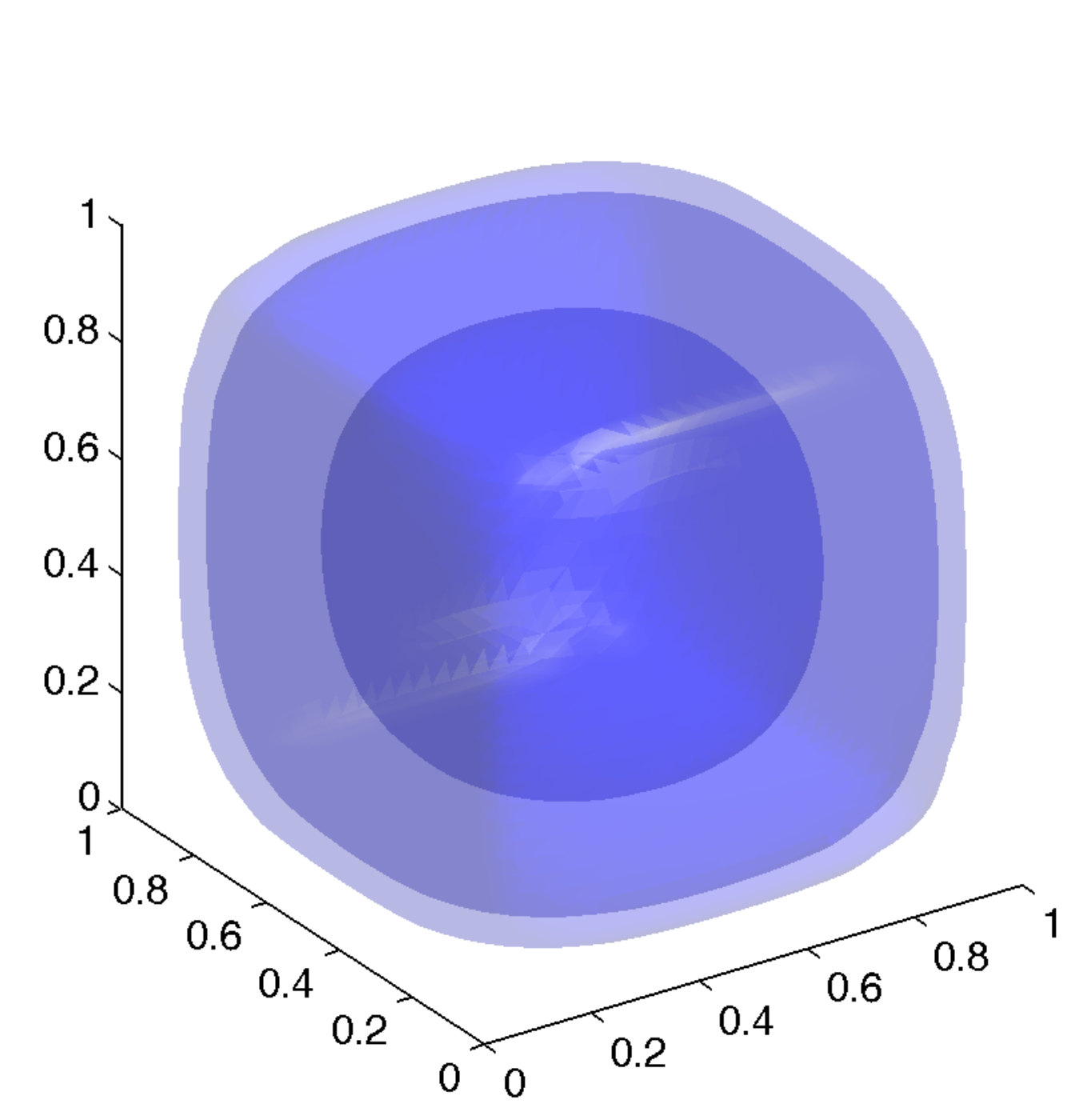}
\end{tabular}
\caption{Surface plots of the level sets of the solution to Example \ref{surfacelevelsetscube} on a $30\times 30\times 30$ grid with the naive finite differences (left) and the $27$-point monotone scheme (right).}
\label{fig:Ex7surfacelevelsetplots}
\end{figure}

\begin{figure}[htp]
\centering
\begin{tabular}{cc}
\includegraphics[width=0.65\textwidth]{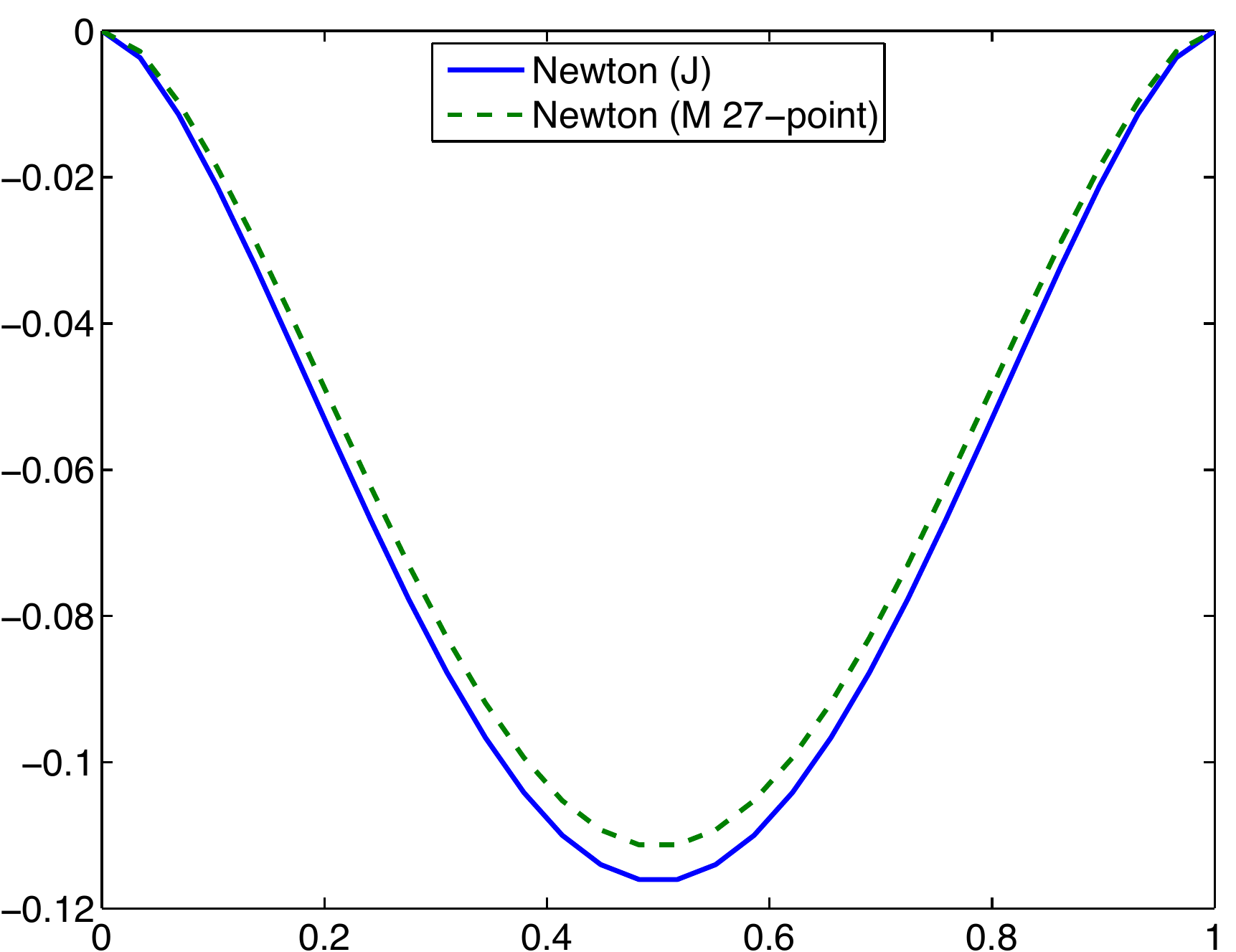}\end{tabular}
\caption{Plot of the curves $t\mapsto u(t,t,t)$ of the solution of Example \ref{surfacelevelsetscube} on a $30\times 30 \times 30$ grid.}
\label{fig:Ex7curve}
\end{figure}

\begin{example}\label{surfacelevelsetssphere}
We consider as well the example with $f \equiv 1$ and $g \equiv 0$ Dirichlet boundary conditions but with a different domain $\Omega = \Omega_1 \cup \Omega_2$ where
\[\Omega_1 = \{(x,y,z)\in\R^3:(x-0.35)^2+(y-0.35)^2+(z-0.5)^2 < 0.3^2\},\]
\[\Omega_2 = \{(x,y,z)\in\R^3:(x-0.65)^2+(y-0.65)^2+(z-0.5).^2 < 0.3^2\}.\]

No exact solution is known. In Figure \ref{fig:Ex8surfacelevelsetplots}, we illustrate some of the surface plots of the level sets $u=c$ of the solution with the standard finite differences and monotone scheme with $c\in\{0,-0.01,-0.02,-0.03,-0.035,-0.039\}$. In this case the zero level set is not convex, with the level sets $u=c$ becoming more convex with smaller values of $c$. In this case the difference between the standard finite differences and monotone scheme is even smaller than in Example \ref{surfacelevelsetscube}, as we can see in Figure \ref{fig:Ex8curve}, where we plot the curve $u(t,t,t)$ with $t \in [0,1]$.
\end{example}

\begin{figure}[htp]
\centering
\begin{tabular}{cc}
\hspace{-1in}\includegraphics[width=0.65\textwidth]{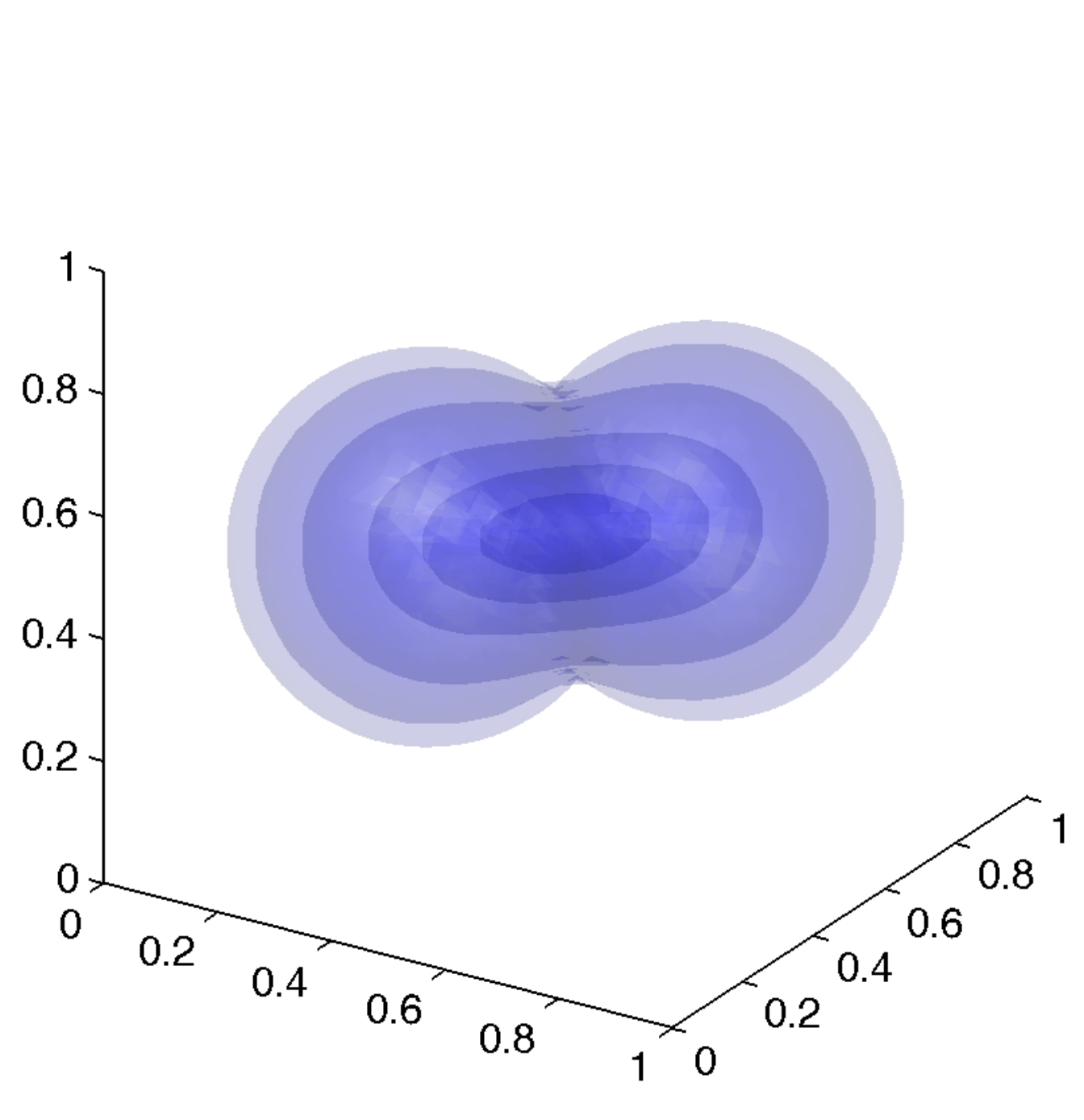} & \includegraphics[width=0.65\textwidth]{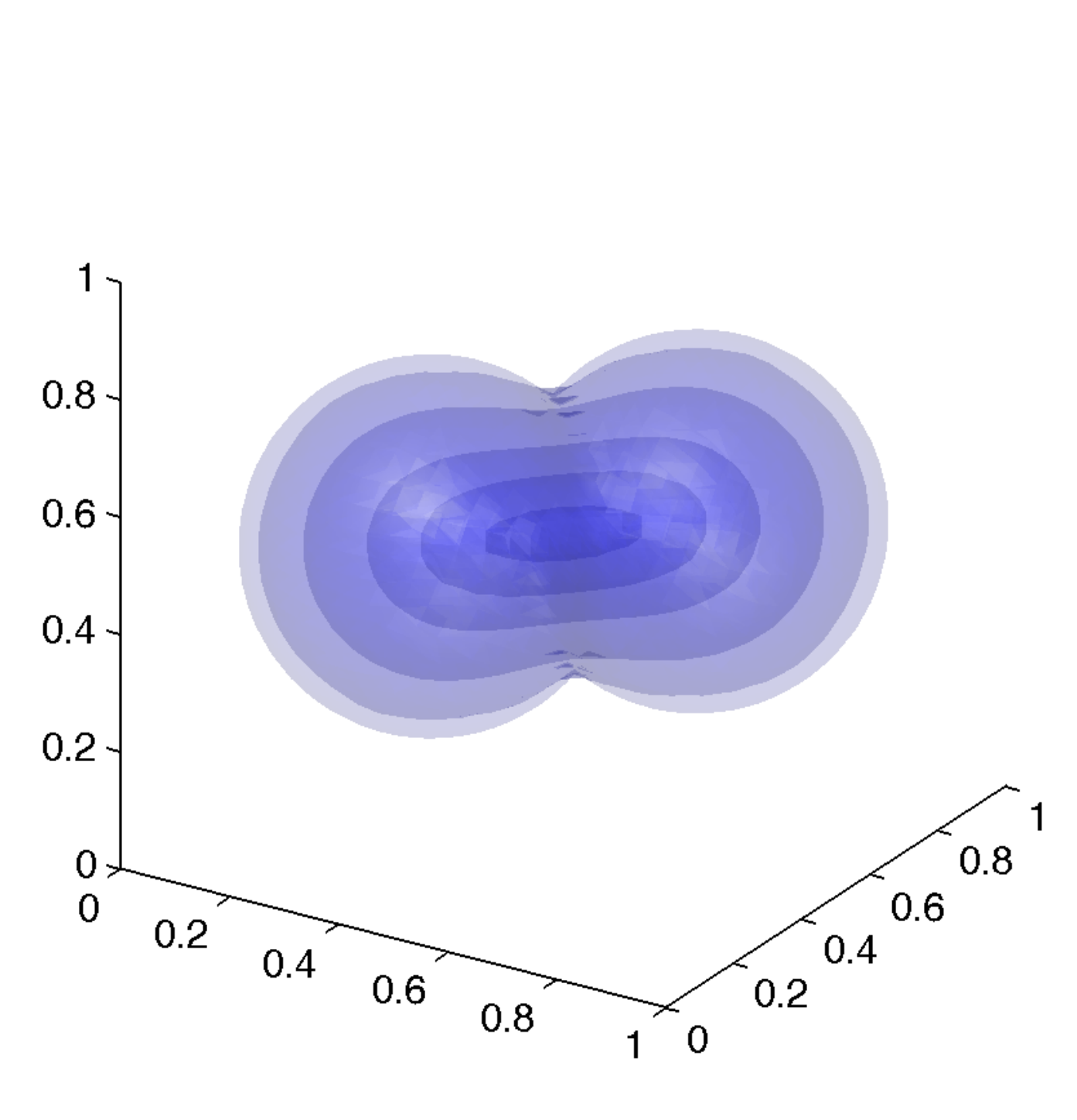}
\end{tabular}
\caption{Surface plots of the level sets of the solution to Example \ref{surfacelevelsetssphere} on a $30\times 30\times 30$ grid with the naive finite differences (left) and the $27$-point monotone scheme (right).}
\label{fig:Ex8surfacelevelsetplots}
\end{figure}

\begin{figure}[htp]
\centering
\begin{tabular}{cc}
\includegraphics[width=0.65\textwidth]{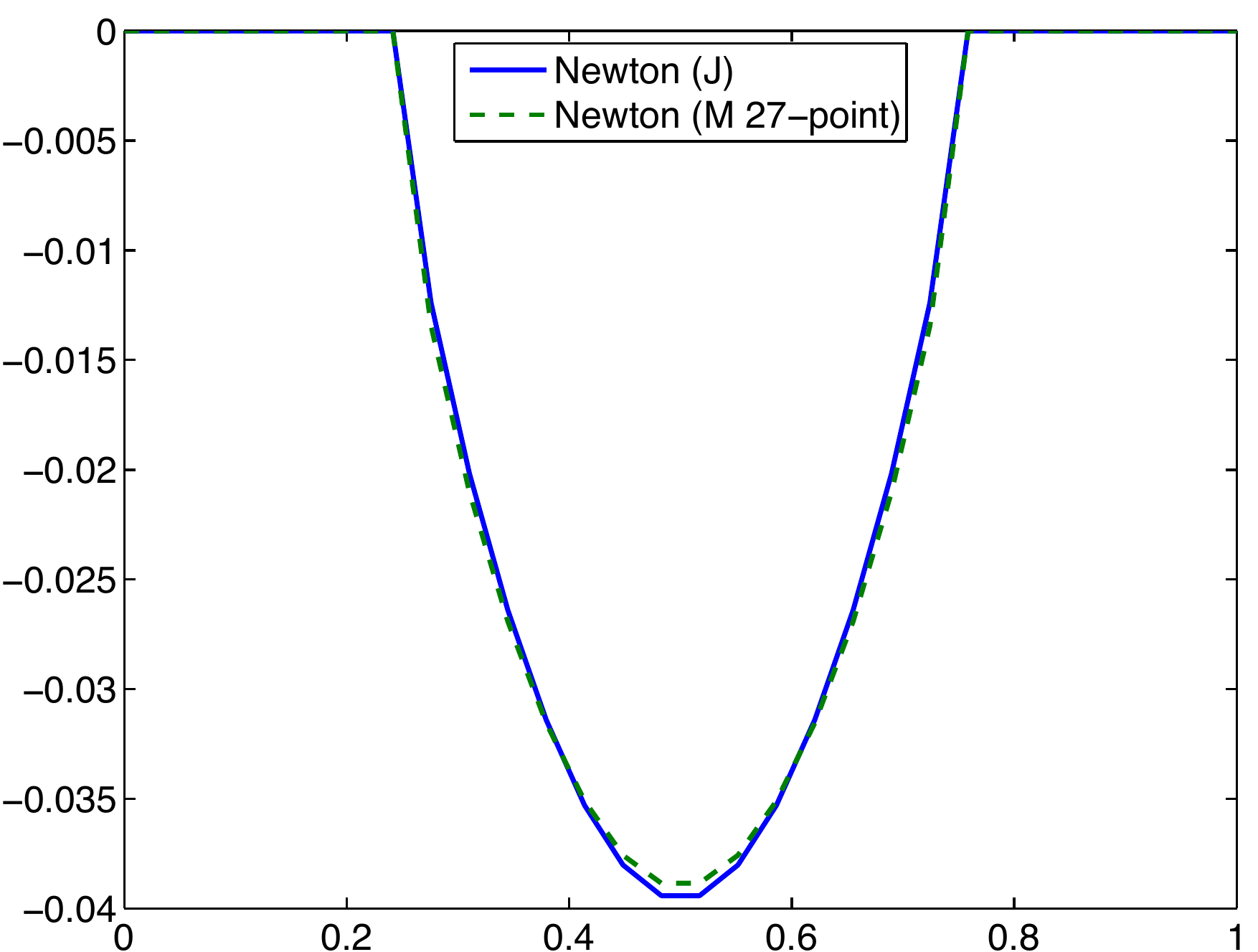}\end{tabular}
\caption{Plot of the curves $t\mapsto u(t,t,t)$ of the solution of Example \ref{surfacelevelsetssphere} on a $30\times 30 \times 30$ grid.}
\label{fig:Ex8curve}
\end{figure}


\section{Conclusions}\label{sec:conclusions}
\vspace{1.2ex}
The 2-Hessian equation is a fully nonlinear Partial Differential Equation which is elliptic provided the solutions are restricted to a convex cone, which we called plane-subharmonic.  It is natural to compare this equation with the Monge-Amp\`{e}re PDE, which is elliptic on the cone of convex functions, and which has been studied numerically   in previous work by two of the authors.  The elliptic 2-Hessian equation is more challenging because the constraints for ellipticity are less restrictive.

We gave two different discretizations for the $2$-Hessian equation in the three-dimensional case: a naive one obtained by simply using standard finite differences to discretize the Hessian and a monotone discretization that takes advantage of a characterization of the operator using a matrix inequality \eqref{monotoneidentity}. The monotone discretization is provably convergent but less accurate, because the monotone discretization required the use of a wide stencil.
 Computational results were provided using exact solutions of varying regularity and shape, from smooth to non differentiable, and from convex to nonconvex.

The naive discretization failed, unless we introduced a mechanism for selecting the correct 2-admissible  (plane-subharmonic) solution.  Once this mechanism was introduced, experimental results on a variety of solutions demonstrated that the method appeared to converge. 
The standard finite difference discretization failed using a standard  parabolic solver.
Two alternative solvers were presented, which  enforced the ``plane-subharmonic'' restriction and proved to work numerically for all the examples considered. Additionally, a Newton solver was also implemented, converging for all examples considered, even for degenerate ones or with singular right-hand sides, whenever initialized with a good initial guess. For smooth examples, we obtained second order convergence.

The monotone discretization, less accurate due the introduction of a directional resolution to make it monotone, is stable and provably convergent.  Numerical examples show that the directional resolution easily dominates the spacial resolution, a natural consequence of the three dimensional setting.

Moreover, one could have implemented filtered schemes, previously introduced in~\cite{FroeseObermanFiltered}, which would provide schemes that are provably convergent but with greater accuracy than the monotone schemes. However, we did not implement them here, since our main goal was to compare the two different discretizations presented and, moreover, the accurate scheme by itself proved to be convergent for all the examples considered, even degenerate ones.

The 2-Hesssian equation is related to the scalar curvature, these are equal up to a constant when the gradient of the function vanishes. A natural extension to the current work is to build schemes for the prescribed scalar curvature of a three dimensional graph.

In this work, we chose the box domain since it is easier to deal with computationally as the boundary conditions are easily implemented. Dealing with more complex boundaries requires additional work. 
It is challenging to obtain higher order at the boundary while maintaining second order directional derivatives. A natural approach would be a combination of filtered schemes at the boundary and multi-scale grids~\cite{ObermanZwiers}. Unstructured grids are another possibility, having been used successfully by one of the authors to solve several fully nonlinear elliptic equations~\cite{FroeseMeshfree}. 

\bibliographystyle{alpha}
\bibliography{TwoHessian}

\end{document}